\documentclass[smallextended]{svjour3}

\smartqed  % flush right qed marks, e.g. at end of proof

\RequirePackage{fix-cm}
\usepackage{graphicx}
\usepackage{psfrag}
\usepackage{subfigure}
\usepackage{url}
\usepackage{color}
\usepackage{cite}
\usepackage{epsfig}
\usepackage{amsmath,amssymb}

\newcommand{\rset}{\mathbb{R}}

\usecounter{algorithmenumi}

\providecommand{\norm}[1]{\lVert#1\rVert}

\begin{document}

\title{Linear convergence  of first order methods for
non-strongly convex  optimization}

\titlerunning{Linear convergence  of first order methods
 for non-strongly convex  optimization}

\author{I. Necoara  \and Yu. Nesterov \and F. Glineur}

\institute{I. Necoara  \at Automatic Control and Systems
Engineering Department, University Politehnica Bucharest, 060042
Bucharest, Romania, \email{ion.necoara@acse.pub.ro.}
\and
Yu. Nesterov, F. Glineur \at Center for Operations Research
and Econometrics, Universite catholique de Louvain,
Louvain-la-Neuve, B-1348, Belgium, \email{\{yurii.nesterov,
francois.glineur\}@uclouvain.be.}
}

\date{Received: date / Accepted: date}

\maketitle

\begin{abstract}
The standard assumption for proving linear convergence of first
order methods for smooth convex optimization is the strong convexity
of the objective function, an assumption which does not hold for
many practical applications. In this paper, we derive linear
convergence rates of several first order methods for solving
smooth non-strongly convex constrained optimization problems, i.e.
involving an objective function with a Lipschitz continuous gradient
that satisfies some relaxed strong convexity condition. In
particular, in the case of  smooth constrained convex optimization,
we provide several relaxations of the strong convexity conditions and
prove that they  are sufficient for getting linear convergence for
several first order methods such as projected gradient, fast
gradient and feasible descent methods. We also provide examples of
functional classes  that satisfy our proposed relaxations of strong
convexity conditions.  Finally, we show that the
proposed relaxed strong convexity conditions cover important
applications ranging from solving  linear systems, Linear
Programming, and dual formulations of linearly constrained convex
problems.
\end{abstract}

%%%%%%%%%%%%%%%%%%%%%%%%%%%%%%%%%%%%%%%%%%%%%%%%%%%%%%%%%%%%%%%%%

\section{Introduction}

\noindent Recently, there emerges a surge of interests in
accelerating first order methods for difficult optimization
problems,  for example  the ones without strong convex objective
function, arising in different applications such as data analysis
\cite{LevLew:10} or machine learning \cite{LiuWri:15}. Algorithms
based on gradient information have proved to be effective in these
settings,  such as  projected gradient and its fast variants
\cite{Nes:04}, stochastic gradient descent \cite{NemJud:09} or
coordinate gradient descent \cite{Wri:15}.

%\vspace{0.1cm}

\noindent For smooth convex programming, i.e. optimization problems
with convex objective function having a Lipschitz continuous
gradient with constant $L_f>0$,  first order methods are converging
sublinearly.  In order to get an $\epsilon$-optimal solution, we
need to perform $\mathcal{O} \left( \frac{L_f}{\epsilon} \right)$ or
even $\mathcal{O} \left( \sqrt{\frac{L_f}{\epsilon}} \right)$ calls
to the oracle \cite{Nes:04}. Typically, for proving linear
convergence of the first order methods  we also need to require
strong convexity for the objective function. Unfortunately, many
practical applications do not have strong convex   objective
function. A new line of analysis, that circumvents these
difficulties, was developed using several notions. For example,
sharp minima type condition for non-strongly convex optimization
problems, i.e. the epigraph of the objective function is a
polyhedron, has been proposed in \cite{BurDen:09,LewPan:98,YanLin:15}. An
error bound property, that estimates the distance to the solution
set from any feasible point by the norm of the proximal residual,
has been analyzed in \cite{LuoTse:92,NecCli:14,WanLin:14}. Finally,
a restricted (also called essential) strong convexity inequality,
which basically imposes a quadratic lower bound on the objective
function, has been derived in \cite{ZhaYin:13,LiuWri:15}. For all
these conditions (sharp minima, error bound or restricted strong
convexity) several gradient-type methods are shown to converge
linearly, see e.g.
\cite{LuoTse:92,LiuWri:15,NecCli:14,WanLin:14,ZhaYin:13}. Several
other papers on linear convergence of first order methods for
non-strongly convex optimization  have appeared recently
\cite{BecSht:15,YanLin:15}. The main goal of this paper is to
develop a framework for finding general functional  conditions for
smooth convex constrained optimization problems that allow us to
prove linear convergence for a broad class of first order methods.

\vspace{2pt}

\noindent \textit{Contributions}: For smooth  convex constrained
optimization,  we show in this paper that some relaxations of the strong
convexity conditions of the objective function are sufficient for obtaining
linear convergence for several first order methods. The most general
relaxation we introduce is a quadratic functional growth condition,
which states that the objective function grows faster than the
squared distance between any feasible point and the optimal set. We
also propose other non-strongly convex conditions, which are more
conservative than the quadratic functional growth condition, and
establish relations  between them. Further, we provide examples of
functional classes  that satisfy our proposed relaxations of strong convexity conditions.
For all these smooth non-strongly
convex constrained optimization problems, we prove that the corresponding
relaxations are sufficient for getting linear convergence for
several  first order methods, such as projected gradient, fast
gradient and feasible descent methods. We also show that the
corresponding linear rates are improved in some cases compared to
the existing results. We also establish necessary and sufficient
conditions for linear convergence of the gradient method. Finally,
we show that the proposed relaxed strong convexity conditions cover
important applications ranging from solving linear systems, Linear
Programming, and dual formulations of linearly constrained convex
problems.

\vspace{2pt}

\noindent \textit{Notations:} We work in the space $\rset^n$
composed by column vectors and $\rset^n_+$ denotes
the non-negative orthant. For $u,v \in \rset^n$ we denote the Euclidean
 inner product $\langle u,v \rangle = u^T v$,
Euclidean norm $\left \| u \right \|=\sqrt{\langle u,u \rangle}$ and
the projection of $u$ onto  convex set $X$ as $\left[u
\right]_X=\arg \min_{x \in X} \|x-u\|$. For  matrix $A \in \rset^{m \times n}$, we denote $\sigma_\text{min}(A)$ the smallest nonzero singular value and $\|A\|$ spectral norm.

%%%%%%%%%%%%%%%%%%%%%%%%%%%%%%%%%%%%%%%%%%%%%%%%%%%%%%%%%%%
%%%%%%%%%%%%%%%%%%%%%%%%%%%%%%%%%%%%%%%%%%%%%%%%%%%%%%%%%%%%5

\section{Problem formulation}
In this paper we consider the  class of convex constrained optimization
problems:
\[ \textbf{(P)}: \quad  f^* = \min_{x \in X} f(x),  \]
where   $X \subseteq \rset^n$ is a simple closed convex set, that is the
projection onto this set is easy,  and $f:X \to \rset$ is a closed convex
function. We further denote by $X^* = \arg \min_{x \in X} f(x)$
the set of optimal solutions of problem (P). We assume throughout
the paper that the optimal set $X^*$  is nonempty and closed and
the optimal value $f^*$ is finite. Moreover, in this paper we assume that the  objective
function is smooth, that is $f$ has Lipschitz continuous gradient with constant
$L_f>0$ on the set $X$:
\begin{align}
\label{lipg} \| \nabla f(x) - \nabla f(y) \| \leq L_f \|x - y\|
\quad \forall x, y \in X.
\end{align}
An immediate consequence of \eqref{lipg} is the following inequality \cite{Nes:04}:
\begin{align}
\label{lipg2}  f(y)  \leq  f(x) + \langle \nabla f(x), y - x \rangle
+  \frac{L_f}{2} \|x - y\|^2 \quad \forall x, y \in X,
\end{align}
while, under convexity of $f$, we also have:
\begin{align}
\label{lipg3} 0 \leq  \langle \nabla f(x) - \nabla f(y), x - y \rangle  \leq  L_f \|x - y\|^2
\quad \forall x, y \in X.
\end{align}
%We denote the class of convex continuously differentiable functions
%with Lipschitz gradient on $X$ by $\mathcal{F}^{}_{L_f} (X)$.

\noindent It is well known that  first order methods are converging sublinearly
on the class of problems whose objective function $f$ has Lipschitz continuous
gradient with constant $L_f$ on the set  $X$, e.g.
convergence rates in terms of
function values of order \cite{Nes:04}:
\begin{equation}
\label{lin_conv_Lf}
\begin{aligned}
& f(x^k) - f^* \leq \frac{L_f \|x^0 - x^*\|^2}{2k} \quad
\text{for projected gradient},  \\
& f(x^k) - f^* \leq \frac{2 L_f \|x^0 - x^*\|^2}{(k+1)^2} \quad \text{for fast gradient},
\end{aligned}
\end{equation}
where $x^k$ is the $k$th iterate generated by the method. Typically,
in order to show linear convergence of first order  methods applied
for solving smooth convex problems, we need to require strong convexity of the objective function.  We recall that $f$ is strongly convex function  on the convex set $X$ with constant $\sigma_f>0$ if the following inequality holds \cite{Nes:04}:
\begin{align}
\label{sc} f(\alpha x + (1 - \alpha) y) \leq \alpha f(x) + (1 -
\alpha) f(y) - \frac{\sigma_f \alpha (1 - \alpha)}{2} \|x - y\|^2
\end{align}
for all $x,y \in X$ and  $\alpha \in [0, \; 1]$. Note that if
$\sigma_f =0$, then $f$ is simply a convex function. We denote by
$\mathcal{S}^{}_{L_f,\sigma_f} (X)$ the class of $\sigma_f$-strongly
convex functions  with an $L_f$-Lipschitz continuous  gradient on
$X$. First order methods are converging linearly on the class
of problems (P) whose objective function $f$ is in
$\mathcal{S}^{}_{L_f,\sigma_f} (X)$, e.g. convergence rates of order
\cite{Nes:04}:
\begin{equation}
\label{lin_conv_Lfsf}
\begin{aligned}
& f(x^k) - f^* \leq \frac{L_f \|x^0 - x^*\|^2}{2} \left( 1 - \frac{\sigma_f}{L_f}\right)^k
\quad \text{for projected gradient},  \\
& f(x^k) - f^* \leq 2\left( f(x^0) - f^* \right) \left( 1 - \sqrt{\frac{\sigma_f}{L_f}}\right)^k \quad \text{for fast gradient}.
\end{aligned}
\end{equation}

\noindent In the case of a differentiable function $f$ with $L_f$-Lipschitz continuous  gradient, each of the following conditions below is equivalent to inclusion $f \in
\mathcal{S}^{}_{L_f,\sigma_f} (X)$~\cite{Nes:04}:
\begin{equation}
\label{sc_equiv}
\begin{aligned}
& f(y) \geq f(x) + \langle \nabla f(x), y-x \rangle +  \frac{\sigma_f}{2} \|x - y \|^2 \qquad \forall x,y \in X,  \\
& \sigma_f \| x -  y\|^2 \leq \langle \nabla f(x) - \nabla f(y), x-y  \rangle  \qquad \forall x, y \in X.
\end{aligned}
\end{equation}
Let us give some properties of smooth  strongly convex functions from the class $\mathcal{S}^{}_{L_f,\sigma_f} (X)$.
Firstly, using the optimality conditions for (P), that is $\langle \nabla f(x^*), y-x^* \rangle \geq 0$ for all $y \in X$ and $x^* \in X^*$,  in the first
inequality in \eqref{sc_equiv} we  get the following relation:
\begin{equation}
\label{sc1}
\begin{aligned}
& f(x) - f^* \geq \frac{\sigma_f}{2} \|x - x^* \|^2 \qquad \forall x
\in X.
\end{aligned}
\end{equation}
Further, the gradient mapping  of a continuous differentiable
function $f$ with Lipschitz gradient in a point $x \in \rset^n$ is
defined as \cite{Nes:04}:
\[ g(x) = L_f \left( x - [x - 1/L_f \nabla f(x)]_X \right), \]
If additionally, the function $f$ has also Lipschitz continuous gradient, then
we obtain a second relation valid for any $f \in
\mathcal{S}^{}_{L_f,\sigma_f} (X)$ \cite{WanLin:14}[Lemma 22]:
\begin{equation}
\label{sc12}
\begin{aligned}
& \frac{\sigma_f}{2} \|x - y \| \leq  \|g(x) - g(y) \| \qquad
\forall x,y \in X.
\end{aligned}
\end{equation}

\noindent However, in many applications the  strong convexity condition
\eqref{sc} or equivalently  one of the conditions \eqref{sc_equiv} cannot be assumed to hold. Therefore, in the next sections we introduce some
non-strongly convex conditions for the objective function $f$ that are
less conservative than strong convexity. These are based on relaxations of
 strong convexity relations \eqref{sc_equiv}--\eqref{sc12}.

%%%%%%%%%%%%%%%%%%%%%%%%%%%%%%%%%%%%%%%%%%%%%%%%%%%%%

\section{Non-strongly convex conditions for a function}
In this section we  introduce   several  functional classes  that
are relaxing the strong convexity properties
\eqref{sc_equiv}--\eqref{sc12} of a function and derive relations
between these classes. More precisely, we observe that strong
convexity relations \eqref{sc_equiv} or \eqref{sc12} are valid for
all $x,y \in X$. We propose in this paper functional classes
satisfying conditions of the form \eqref{sc_equiv} or \eqref{sc12}
that hold for some particular choices of $x$  and $y$, or satisfying
simply the condition \eqref{sc1}.
%Instead of assuming that  inequalities
%\eqref{sc_equiv} or \eqref{sc12}  hold for all $x,y \in X$, we impose that such a
%condition is valid only for some particular choices of $x$  and
%$y$.

\subsection{Quasi-strong convexity}
The first non-strongly convex  relaxation we introduce is based on
choosing a particular value for $y$ in the first strong convexity
inequality in \eqref{sc_equiv}, that is  $y = \bar{x}  \equiv
[x]_{X^*}$ (recall that $[x]_{X^*}$ denotes the projection of $x$
onto the optimal set $X^*$ of convex problem (P)):

\begin{definition}
Continuously differentiable  function $f$  is called
\textit{quasi-strongly convex}  on set $X$ if there exists a
constant $\kappa_f >0$ such that for any $x \in X$ and  $\bar{x} =
[x]_{X^*}$   we have:
 \begin{align}
  \label{wsc_basic}
  f^* \geq f(x) + \langle \nabla f(x), \bar{x} - x \rangle + \frac{\kappa_f}{2}
  \|  x - \bar{x} \|^2 \quad \forall x \in X.
 \end{align}
\end{definition}

\noindent   Note that inequality \eqref{wsc_basic} alone does not
even imply  convexity of  function $f$.  Moreover, our definition of
 quasi-strongly convex functions does not ensure uniqueness of the
optimal solution of problem (P) and does not require $f$ to have Lipschitz continuous gradient.
We denote the class of convex functions   with Lipschitz
continuous gradient with constant $L_f$ in \eqref{lipg} and
satisfying the quasi-strong convexity property with constant
$\kappa_f$ in \eqref{wsc_basic} by
$q\mathcal{S}_{L_f,\kappa_f}^{}(X)$. Clearly, for strongly convex
functions with constant $\kappa_f$,   from the first condition in \eqref{sc_equiv} with $y =
x^* \in X^*$, we observe that the following inclusion hold:
\begin{align}
\label{inclusion_sc} \mathcal{S}^{}_{L_f,\kappa_f} (X) \subseteq
q\mathcal{S}_{L_f,\kappa_f}^{}(X).
\end{align}
Moreover, combining the inequalities  \eqref{lipg2} and
\eqref{wsc_basic}, we obtain that  the condition number of  objective function $f \in
q\mathcal{S}_{L_f,\kappa_f}^{}(X)$, defined as $\mu_f= \kappa_f/L_f$,
satisfies:
\begin{align}
\label{cond_P}
0 < \mu_f  \leq 1.
\end{align}

\noindent We will derive  below other functional classes  that are
related to our newly introduced class of quasi-strongly convex
functions $q\mathcal{S}_{L_f,\kappa_f}^{}(X)$.

%%%%%%%%%%%%%%%%%%%%%%%%%%%%%%%%%%%%%%%%%%%%%%%%%%%%%%%%%%%%%%%%%%%%%%%%%%%%%%%%%%%%

\subsection{Quadratic under-approximation}
Let us  define the class of functions   satisfying a quadratic
under-approximation  on the set $X$, obtained from relaxing the
first inequality in \eqref{sc_equiv} by choosing $y=x$ and $x =
\bar{x} \equiv[x]_{X^*}$:

\begin{definition}
Continuously differentiable  function $f$  has a \textit{quadratic
under-approximation}   on $X$ if there exists a constant $\kappa_f
>0$ such that for any $x \in X$ and  $\bar{x} = [x]_{X^*}$  we have:
\begin{align}
\label{qap}
f(x) \geq  f^* +  \langle \nabla f(\bar{x}), x - \bar{x} \rangle  +
  \frac{\kappa_f}{2} \|x - \bar{x} \|^2 \quad \forall x \in X.
\end{align}
\end{definition}

\noindent We denote the class of convex  functions with Lipschitz
continuous gradient and satisfying the quadratic under-approximation
property \eqref{qap} on $X$ by $\mathcal{U}_{L_f,\kappa_f}^{}(X)$.
Then, we have the following inclusion:

\begin{theorem}
\label{th_qap=wsc} Inequality \eqref{wsc_basic} implies inequality \eqref{qap}. Therefore, the following inclusion holds:
\begin{align}
\label{inclusion_qap} q\mathcal{S}_{L_f,\kappa_f}^{}(X)  \subseteq
\mathcal{U}_{L_f,\kappa_f}(X).
\end{align}
\end{theorem}

\begin{proof}
Let $f \in q\mathcal{S}_{L_f,\kappa_f}^{}(X)$. Since $f$ is convex function, it satisfies
the  inequality  \eqref{qap} with some constant $\kappa_f(0) \geq 0$, i.e.:
\begin{align}
\label{qap0}
f(x) \geq  f^* +  \langle \nabla f(\bar{x}), x - \bar{x} \rangle  +  \frac{\kappa_f(0)}{2} \|x - \bar{x} \|^2.
\end{align}
Using first order Taylor approximation in the integral form we have:
\begin{align*}
& f(x)   = f(\bar{x}) + \int_0^1 \langle \nabla f(\bar{x} + \tau(x - \bar{x})), x -\bar{x} \rangle d\tau \\
& = f(\bar{x}) +  \int_0^1 \frac{1}{\tau} \langle \nabla f(\bar{x} + \tau(x - \bar{x})), \tau(x -\bar{x}) \rangle d\tau \\
& \overset{\eqref{wsc_basic} \; \text{in} \; \bar{x} + \tau (x - \bar{x})}{\geq}  f(\bar{x}) +  \int_0^1 \frac{1}{\tau} \left( f(\bar{x} + \tau(x - \bar{x})) - f(\bar{x}) + \frac{\kappa_f}{2} \|\tau(x - \bar{x}) \|^2  \right) d\tau \\
&\overset{\eqref{qap0}}{\geq} \! f(\bar{x}) \!+\!\! \int_0^1 \! \frac{1}{\tau} \! \left(\! \langle \nabla f(\bar{x}), \tau(x -\bar{x}) \rangle + \frac{\kappa_f(0)}{2} \| \tau(x - \bar{x}) \|^2 \!\right)\! +\! \frac{1}{\tau} \frac{\kappa_f}{2} \|\tau(x - \bar{x}) \|^2  d\tau \\
& = f(\bar{x}) + \int_0^1  \langle \nabla f(\bar{x}), x -\bar{x} \rangle + \frac{\tau \kappa_f(0)}{2} \|x - \bar{x} \|^2 + \frac{\tau \kappa_f}{2} \|x - \bar{x} \|^2  d\tau\\
&=  f(\bar{x}) + \langle \nabla f(\bar{x}), x -\bar{x} \rangle + \frac{\kappa_f(0) + \kappa_f}{2} \cdot \frac{1}{2} \|x - \bar{x} \|^2.
\end{align*}
If we denote $\kappa_f(1) = \frac{\kappa_f(0) + \kappa_f}{2}$, then
we get that inequality  \eqref{qap0} also holds for $\kappa_f(1)$.
Repeating the same argument as above for  $f \in
q\mathcal{S}_{L_f,\kappa_f}^{}(X)$ and satisfying  \eqref{qap0}  for
$\kappa_f(1)$ we get that inequality  \eqref{qap0} also holds for
$\kappa_f(2) = \frac{\kappa_f(1) + \kappa_f}{2} = \frac{\kappa_f(0)
+ 3 \kappa_f}{4}$. Iterating this procedure we obtain that after $t$
steps:
\[ \kappa_f(t) = \frac{\kappa_f(t-1) + \kappa_f}{2} =
\frac{\kappa_f(0) + (2^t -1) \kappa_f}{2^t} \to \kappa_f \quad
\text{as} \quad t \to \infty.  \] Since after any $t$ steps the
inequality  \eqref{qap0} holds with $\kappa_f(t)$, using  continuity
of $\kappa_f(t)$ in \eqref{qap0} we obtain \eqref{qap}. This proves
our statement. \qed
\end{proof}
Moreover, combining the inequalities  \eqref{lipg2} and \eqref{qap},
we obtain that  the condition number of  objective function $f \in \mathcal{U}_{L_f,\kappa_f}(X)$, defined as $\mu_f = \kappa_f/L_f$, satisfies:
\begin{align}
\label{cond_Pqap}
0 < \mu_f  \leq 1.
\end{align}

%%%%%%%%%%%%%%%%%%%%%%%%%%%%%%%%%%%%%%%%%%%%%%%%%%%%%%%%%%%%%%%%%%%%%%%%%%

\subsection{Quadratic gradient growth}
Let us define the class of functions  satisfying a bound  on the
variation  of gradients over the set  $X$. It is obtained by
relaxing the second inequality in \eqref{sc_equiv} by choosing  $y=
\bar{x} \equiv [x]_{X^*}$:

\begin{definition}
Continuously differentiable  function $f$  has a \textit{quadratic gradient growth}   on set  $X$ if there exists a constant $\kappa_f >0$ such that for any $x \in X$ and  $\bar{x} = [x]_{X^*}$  we have:
\begin{align}
\label{gvar} \langle  \nabla f(x) - \nabla f(\bar{x}), x - \bar{x}
\rangle \geq \kappa_f \| x - \bar{x}\|^2 \quad \forall x \in X.
\end{align}
\end{definition}

\noindent Now, let us denote the class of convex  differentiable
functions with Lipschitz gradient  and satisfying the  quadratic
gradient growth \eqref{gvar} by $\mathcal{G}_{L_f,\kappa_f}(X)$. In
\cite{ZhaYin:13} the authors analyzed a similar class of objective functions,
but for unconstrained optimization problems, that is $X = \rset^n$, which
was called  \textit{restricted strong convexity} and was defined
as: there exists  a constant $\kappa_f>0$ such that
$\langle \nabla f(x), x - \bar{x} \rangle \geq \kappa_f \| x -
\bar{x}\|^2 $ for all $x \in \rset^n$.
%For smooth convex unconstrained problems whose objective function satisfies the
%restricted strong convexity inequality the authors in
%\cite{ZhaYin:13} show linear convergence for several gradient-type
%methods.
An immediate consequence of Theorem \ref{th_qap=wsc} is the
following inclusion:
\begin{theorem}
\label{th_var=wsc}
Inequality \eqref{wsc_basic} implies inequality \eqref{gvar}. Therefore, the following inclusion holds:
\begin{align}
\label{inclusion_var} q\mathcal{S}_{L_f,\kappa_f}^{}(X)  \subseteq
\mathcal{G}_{L_f,\kappa_f}(X).
\end{align}
\end{theorem}

\begin{proof}
If $f \in q\mathcal{S}_{L_f,\kappa_f}^{}(X)$, then $f$ satisfies the
inequality \eqref{wsc_basic}. From Theorem \ref{th_qap=wsc} we also
have that $f$  satisfies inequality \eqref{qap}.  By adding the two inequalities \eqref{wsc_basic} and
\eqref{qap} in $x$ we get:
\begin{align}
\label{inclusion_qap_cor}
\langle  \nabla f(x) - \nabla f(\bar{x}), x - \bar{x} \rangle \geq \kappa_f \| x - \bar{x}\|^2 \quad \forall x \in X,
\end{align}
which proves  that inequality \eqref{gvar} holds. \qed
\end{proof}

\noindent We prove below that \eqref{wsc_basic} or \eqref{qap} alone and convexity of $f$ implies \eqref{gvar} with constant  $\kappa_f/2$. Indeed, let us assume for example that \eqref{qap} holds, then we~have:
\begin{align*}
f(x)  & \geq  f^* +  \langle \nabla f(\bar{x}), x - \bar{x} \rangle  +  \frac{\kappa_f}{2} \|x - \bar{x} \|^2\\
& \geq f(x) + \langle \nabla f(x), \bar{x} - x \rangle +  \langle \nabla f(\bar{x}), x - \bar{x} \rangle  +  \frac{\kappa_f}{2} \|x - \bar{x} \|^2 \\
& = f(x) + \langle \nabla f(\bar{x}) - \nabla f(x), x - \bar{x} \rangle  +  \frac{\kappa_f}{2} \|x - \bar{x} \|^2,
\end{align*}
which leads to \eqref{gvar} with constant  $\kappa_f/2$.
Combining the inequalities  \eqref{lipg3} and \eqref{gvar}, we
obtain that  the condition number  of objective function $f \in \mathcal{G}_{L_f,\kappa_f}(X)$,
satisfies:
\begin{align}
\label{cond_Pqvar}
0 < \mu_f \leq 1.
\end{align}

\begin{theorem}
\label{th_qap=var} Inequality \eqref{gvar} implies inequality \eqref{qap}. Therefore, the following inclusion holds:
\begin{align}
\label{inclusion_qap_var} \mathcal{G}_{L_f,\kappa_f}^{}(X)
\subseteq \mathcal{U}_{L_f,\kappa_f}(X).
\end{align}
\end{theorem}

\begin{proof}
Let $f \in \mathcal{G}_{L_f,\kappa_f}^{}(X)$, then from first
order Taylor approximation in the integral form we get:
\begin{align*}
f(x) & = f(\bar{x}) + \int_0^1 \langle \nabla f(\bar{x} + t(x-\bar{x})), x - \bar{x}  \rangle dt \\
& = f(\bar{x}) + \langle \nabla f(\bar{x}), x - \bar{x} \rangle +  \int_0^1 \langle \nabla f(\bar{x} + t(x-\bar{x})) - \nabla f(\bar{x}), x - \bar{x}  \rangle dt \\
& = \! f(\bar{x}) + \langle \nabla f(\bar{x}), x - \bar{x} \rangle + \! \int_0^1 \!  \frac{1}{t}\langle \nabla f(\bar{x} + t(x-\bar{x})) - \nabla f(\bar{x}), t(x - \bar{x})  \rangle dt \\
& \overset{\eqref{gvar}}{\geq}  f(\bar{x}) + \langle \nabla f(\bar{x}), x - \bar{x} \rangle + \int_0^1 \frac{1}{t} \kappa_f \|t(x -\bar{x})\|^2   dt \\
& = f(\bar{x}) + \langle \nabla f(\bar{x}), x - \bar{x} \rangle + \frac{\kappa_f}{2} \|x -\bar{x}\|^2,
\end{align*}
where we used that $[\bar{x} + t(x - \bar{x})]_{X^*} = \bar{x}$ for any $t \in [0, \ 1]$. This chain of inequalities proves that $f$ satisfies inequality \eqref{qap} with the same constant  $\kappa_f$. \qed
\end{proof}

%%%%%%%%%%%%%%%%%%%%%%%%%%%%%%%%%%%%%%%%%%%%%%%%%%%%%%%%%%%%%%5

\subsection{Quadratic functional growth}
We further define the class of  functions satisfying a quadratic
functional growth property on the set $X$. It shows that the
objective function grows faster than the squared distance between
any feasible point and the optimal set. More precisely,  since
$\langle \nabla f(x^*), y-x^* \rangle \geq 0$ for all $y \in X$ and
$x^* \in X^*$, then using this relation and choosing  $y =x$ and $ x
= \bar{x} \equiv [x]_{X^*}$  in the first inequality in
\eqref{sc_equiv},  we  get a relaxation of this strong convexity
condition  similar to inequality~\eqref{sc1}:

\begin{definition}
Continuously differentiable  function $f$  has a \textit{quadratic functional growth} on   $X$ if there exists a constant $\kappa_f >0$ such that for any $x \in X$ and  $\bar{x} = [x]_{X^*}$  we have:
\begin{align}\label{wsc}
f(x) - f^* \geq \frac{\kappa_f}{2} \|x - \bar{x} \|^2 \quad
\forall x \in X.
\end{align}
\end{definition}

\noindent Since the above  quadratic functional growth inequality is
given in $\bar{x}$, this does not mean that $f$  grows everywhere
faster than the quadratic function $\kappa_f/2 \|x - \bar{x} \|^2$.
We denote the class of convex  differentiable functions with
Lipschitz  continuous gradient and satisfying the quadratic
functional growth \eqref{wsc} by $\mathcal{F}_{L_f,\kappa_f}(X)$. We
now derive inclusion relations between the functional classes we
have introduced so far:

\begin{theorem}
\label{th_qap_sog}
The following chain of implications are valid:
\[ \eqref{sc_equiv} \Rightarrow  \eqref{wsc_basic} \Rightarrow \eqref{gvar} \Rightarrow \eqref{qap} \Rightarrow \eqref{wsc}. \]
Therefore,  the following inclusions hold:
\begin{align}
\label{inclusion_qap_sog} \mathcal{S}_{L_f,\kappa_f}^{}(X) \subseteq
q\mathcal{S}_{L_f,\kappa_f}^{}(X)  \subseteq
\mathcal{G}^{}_{L_f,\kappa_f}(X) \subseteq
\mathcal{U}_{L_f,\kappa_f}(X) \subseteq
\mathcal{F}_{L_f,\kappa_f}^{}(X).
\end{align}
\end{theorem}

\begin{proof}
From the optimality conditions for problem (P) we have
$\langle \nabla f(\bar{x}), x - \bar{x} \rangle  \geq 0 $ for
all $x \in X$. Then,   for any objective function $f$ satisfying
\eqref{qap}, i.e. $f \in \mathcal{U}^{}_{L_f,\kappa_f}(X)$, we
also have \eqref{wsc}. In
conclusion,  from previous derivations, \eqref{inclusion_sc} and
Theorems \ref{th_var=wsc} and \ref{th_qap=var} we obtain our chain
of inclusions. \qed
\end{proof}

\noindent Let us define the  condition number of objective function $f \in \mathcal{F}_{L_f,\kappa_f}^{}(X)$
as $\mu_f = \frac{\kappa_f}{L_f}$. If the feasible set $X$
is unbounded, then combining \eqref{lipg2} with \eqref{wsc} and
considering $\|x - \bar{x}\| \to \infty$, we conclude that:
\begin{align}
\label{cond_Pwsc1}
0 < \mu_f \leq 1.
\end{align}
However, if the feasible set $X$ is bounded,  we may have
$\kappa_f^{s} \gg L_f$,  provided that $\|\nabla f(\bar{x})\|$ is
large, and thus the condition number might be greater than~$1$:
\begin{align}
\label{cond_Pwsc2}
\mu_f  \geq 1.
\end{align}
Moreover, from the inclusions given by  Theorem \ref{th_qap_sog} we conclude that:
\[     \mu_f(\mathcal{S}) \; \leq \;  \mu_f(q\mathcal{S})  \; \leq \;  \mu_f(\mathcal{G})   \;
\leq \;  \mu_f(\mathcal{U})  \; \leq \; \mu_f(\mathcal{F}).  \]

\noindent Let us denote the projected gradient step from $x$ with:
\[ x^+ = [x - 1/L_f \nabla f(x)]_X, \]
and its projection onto the optimal set $X^*$ with $\bar{x}^+
=[x^+]_{X^*}$. Then, we will show that if $x^+$ is closer to $X^*$
than $x$, then the objective function $f$ must satisfy the quadratic functional
growth \eqref{wsc}:
\begin{theorem}
\label{th_necesary_linconv} Let $f$ be a convex function with
Lipschitz continuous gradient with constant $L_f$. If there exists some positive
constant $\beta <1$ such that the following inequality holds:
\[ \| x^+  -  \bar{x}^+ \| \leq \beta \|x - \bar{x} \| \quad \forall x \in X,  \]
then $f$  satisfies the quadratic functional
growth \eqref{wsc} on $X$ with the constant $\kappa_f = L_f(1 - \beta)^2$.
\end{theorem}

\begin{proof}
On the one hand,  from triangle inequality for the projection we
have:
\[  \|x - \bar{x} \| \leq \| x  -  \bar{x}^+ \| \leq \|x  - x^+\| + \| x^+ -
 \bar{x}^+ \|.  \]
Combining this relation with the condition from the theorem, that is
$\| x^+ -  \bar{x}^+ \| \leq \beta \|x - \bar{x} \|$, we get:
\begin{align}
\label{decrease_ball} (1-\beta) \|x - \bar{x} \|  \leq \|x  - x^+\|.
\end{align}
On the other hand, we note that $x^+$ is the optimal solution of the
problem:
\begin{align}
\label{op_g} x^+ = \arg \min_{z \in X} \left[f(x) + \langle \nabla f(x),
z-x \rangle + \frac{L_f}{2} \|z - x\|^2 \right].
\end{align}
From \eqref{lipg2} we have:
\[ f(x^+)  \leq f(x) + \langle \nabla f(x), x^+ - x
\rangle + \frac{L_f}{2} \| x^+ - x\|^2 \] and combining with    the
optimality conditions of \eqref{op_g} in $x$, that is $\langle \nabla f(x) + L_f(x^+ - x), x -x^+
\rangle \leq 0$, we get the following decrease
in terms of the objective function:
\begin{align}
\label{decrease_g} f(x^+) &  \leq f(x) - \frac{L_f}{2} \| x^+ -
x\|^2.
\end{align}
Finally, combining \eqref{decrease_ball} with \eqref{decrease_g},
and using  $f(x^+) \geq f^*$, we get our statement. \qed
\end{proof}

%\noindent  In \cite{ZhaYin:13} the authors show that their restricted secant inequality for the unconstrained problem (P) (i.e. $X = \rset^n$),  which imposes a positive lower bound on the average curvature between any point and the solution set, satisfies our second order growth property \eqref{wsc}.

%%%%%%%%%%%%%%%%%%%%%%%%%%%%%%%%%%%%%%%%%%%%%%%%%%%%%%%%%%%%%%%%%%%%%

\subsection{Error bound property}
Let us recall the gradient mapping  of a continuous differentiable
function $f$ with Lipschitz continuous gradient in a point $x \in
\rset^n$: $g(x) = L_f (x - x^+)$, where  $x^+ = [x - 1/L_f \nabla
f(x)]_X$ is the projected gradient step from $x$. Note that $g(x^*)
=0$ for all $x^* \in X^*$. Moreover,  if $X = \rset^n$, then $g(x) = \nabla
f(x)$.  Recall that the main property of  the gradient mapping for
convex objective functions with Lipschitz continuous gradient of constant $L_f$
is given by the following inequality \cite{Nes:04}[Theorem 2.2.7]:
\begin{align}
\label{gmap} f(y) \geq f(x^+) + \langle g(x), y-x \rangle +
\frac{1}{2L_f} \|g(x)\|^2 \quad \forall y \in X \quad \text{and}
\quad x \in \rset^n.
\end{align}
Taking $y= \bar{x}$ in \eqref{gmap} and using that $f(x^+) \geq
f^*$, we get the  simpler inequality:
\begin{align}
\label{gmap1}
\langle g(x), x - \bar{x} \rangle \geq \frac{1}{2L_f} \|g(x)\|^2 \quad \forall x \in \rset^n.
\end{align}

\noindent In  \cite{LuoTse:92} Tseng
introduced an error bound condition that estimates the distance to the solution set from any feasible point by the norm of the proximal residual: there exists a constant $\kappa>0$ such that
$\| x - \bar{x}\| \leq \kappa \|x - [x -\nabla f(x)]_X \|$ for all $x \in X$. This notion was  further extended and analyzed in \cite{NecCli:14,WanLin:14}.  Next, we define an error bound type condition, obtained from the relaxation of the strong convex inequality \eqref{sc12} for the particular choice $y = \bar{x} \equiv [x]_{X^*}$.

\begin{definition}
The continuous differentiable  function $f$  has a \textit{global
error bound}  on  $X$ if there exists a constant
$\kappa_f>0$ such that for any $x \in X$ and  $\bar{x} =
[x]_{X^*}$  we have:
\begin{align}
\label{geb} \|g(x) \|  \geq \kappa_f \| x - \bar{x}\| \quad
\forall x \in X.
\end{align}
\end{definition}

\noindent  We denote the class of convex functions with Lipschitz continuous
gradient and satisfying the error bound  \eqref{geb} by $\mathcal{E}_{L_f,\kappa_f}^{}$.  Let us define the condition number of the objective function $f \in
\mathcal{E}_{L_f,\kappa_f}^{}(X)$ as $\mu_f=
\frac{\kappa_f}{L_f} $.  Combining  inequality \eqref{gmap1} and
\eqref{geb}  we conclude that the condition number satisfies the
inequality:
\begin{align}
\label{cond_Pgeb}
0< \mu_f \leq 2.
\end{align}
However, for  the unconstrained case, i.e. $X = \rset^n$ and $\nabla
f(\bar{x}) =0$, from \eqref{lipg} and \eqref{geb}  we get $0<
\mu_f \leq 1$. We now determine relations between the quadratic functional growth condition and the error bound condition.

\begin{theorem}
\label{th_sog=eb} Inequality \eqref{geb} implies inequality
\eqref{wsc} with constant $\mu_f \cdot \kappa_f$.  Therefore,  the
following inclusion holds for the functional class
$\mathcal{E}_{L_f,\kappa_f}^{}(X)$:
\begin{align}
\label{inclusion_sog=eb} \mathcal{E}_{L_f,\kappa_f}^{}(X)
\subseteq \mathcal{F}^{}_{L_f,\mu_f \cdot \kappa_f}(X).
\end{align}
\end{theorem}

\begin{proof}
Combining  \eqref{decrease_g} and \eqref{geb}  we obtain:
\begin{align*}
\kappa_f^2 \| x - \bar{x} \|^2  \leq  \|g(x)\|^2 \leq  2 L_f (f(x) - f(x^+)) \leq 2L_f (f(x) - f^*)   \quad \forall x \in X.
\end{align*}
In conclusion, inequality \eqref{wsc} holds with the constant $
\frac{\kappa_f^2}{L_f} =  \mu_f  \cdot \kappa_f$,
where we recall $\mu_f = \kappa_f/L_f$. This also proves the inclusion:
$\mathcal{E}^{}_{L_f,\kappa_f} (X) \subseteq
\mathcal{F}_{L_f,\mu_f \cdot \kappa_f}^{}(X)$.    \qed
\end{proof}

\begin{theorem}
\label{th_eb=sog}
Inequality \eqref{wsc} implies inequality \eqref{geb} with constant $\frac{1}{1 + \mu_f + \sqrt{1 + \mu_f}} \cdot \kappa_f$. Therefore,  the following inclusion holds for the functional class $\mathcal{F}_{L_f,\kappa_f}^{}(X)$:
\begin{align}
\label{inclusion_eb=sog} \mathcal{F}_{L_f,\kappa_f}^{}(X) \subseteq
\mathcal{E}^{}_{L_f,\frac{1}{1 + \mu_f + \sqrt{1 + \mu_f}} \cdot
\kappa_f}(X).
\end{align}
\end{theorem}

\begin{proof}
From the gradient  mapping  property \eqref{gmap} evaluated at the point
$y = \bar{x}^+ \equiv [x^+]_{X^*}$, we  get:
\begin{align*}
f^* & \geq f(x^+) + \langle g(x), \bar{x}^+ - x \rangle + \frac{1}{2L_f} \|g(x)\|^2 \\
& = f(x^+)  + \langle g(x), \bar{x}^+ - x^+ \rangle - \frac{1}{2L_f} \|g(x)\|^2.
\end{align*}
Further, combining the previous inequality and \eqref{wsc}, we obtain:
\[ \langle g(x), x^+ - \bar{x}^+ \rangle + \frac{1}{2L_f} \|g(x)\|^2 \geq f(x^+) - f^* \geq \frac{\kappa_f}{2} \|x^+ - \bar{x}^+ \|^2.  \]
Using Cauchy-Schwartz inequality for the scalar product and then rearranging the terms we obtain:
\[  \frac{1}{2 L_f} \left( \|g(x)\| + L_f \|x^+ - \bar{x}^+\|  \right)^2 \geq \frac{\kappa_f+ L_f}{2} \|x^+ - \bar{x}^+\|^2  \]
or equivalently
\[  \|g(x)\| + L_f \|x^+ - \bar{x}^+\| \geq \sqrt{L_f(\kappa_f + L_f)} \|x^+ - \bar{x}^+\|.  \]
We conclude that:
\[  \|g(x)\|  \geq \left( \sqrt{L_f(\kappa_f + L_f)} -L_f \right) \|x^+ - \bar{x}^+\|. \]
Since $$\|x - \bar{x}\| \leq \|x - \bar{x}^+\|  \leq \| x - x^+\| +
\|x^+ - \bar{x}^+\| = \frac{1}{L_f} \|g(x)\| + \|x^+ -
\bar{x}^+\|,$$ then we obtain:
\[  \|g(x)\|  \geq \left( \sqrt{L_f(\kappa_f + L_f)} -L_f \right) \left( \|x - \bar{x}\| - \frac{1}{L_f} \|g(x)\| \right).  \]
After simple manipulations and using that $\mu_f = \kappa_f/L_f$, we
arrive at:
\[  \|g(x)\|  \geq \frac{\kappa_f}{1 + \mu_f + \sqrt{1 + \mu_f}}  \| x - \bar{x}\|, \]
which shows that inequality \eqref{geb} is valid for the constant $\frac{1}{1 + \mu_f + \sqrt{1 + \mu_f}} \cdot \kappa_f$. \qed
\end{proof}

\noindent Note that the  functional classes  we have introduced
previously were obtained by relaxing the strong convexity
inequalities \eqref{sc_equiv}--\eqref{sc12} for some particular
choices of $x$  and  $y$. The reader can find other favorable
examples of relaxations of strong convexity inequalities and we
believe that this paper opens an window of opportunity for
algorithmic research in non-strongly convex optimization settings.
In the next section we provide concrete examples of objective
functions that can be found in  the functional classes  introduced
above.

%%%%%%%%%%%%%%%%%%%%%%%%%%%%%%%%%%%%%%%%%%%%%%%%%%%%%%%%%%%%
%%%%%%%%%%%%%%%%%%%%%%%%%%%%%%%%%%%%%%%%%%%%%%%%%%%%%%%%%%%%%

\section{Functional classes in  $q\mathcal{S}_{L_f,\kappa_f}^{}(X)$, $\mathcal{G}_{L_f,\kappa_f}^{}(X)$ and $\mathcal{F}_{L_f,\kappa_f}^{}(X)$}
\label{sec_gAx}
We now provide examples of structured convex optimization problems  whose objective function satisfies one of our relaxations of strong convexity conditions  that we have introduced in the previous sections.  We start first recalling some error bounds for the solutions of a system of linear equalities and inequalities.   Let $A\in
\rset^{p \times n}$, $C \in \rset^{m \times n}$ and the arbitrary
norms $\norm{\cdot}_{\alpha}$ and $\norm{\cdot}_{\beta}$ in
$\rset^{m+p}$ and $\rset^n$. Given the nonempty polyhedron:
\[ \mathcal{P} =
\{x \in \rset^n: \; Ax = b, \;  Cx \le d \}, \] then there exists  a
constant $\theta(A,C)>0$ such that Hoffman inequality holds (for a
proof of the Hoffman inequality see e.g. \cite{Hof:52,WanLin:14}):
$$\|x -\bar{x}\| \le \theta(A,C) \left\| \begin{matrix} Ax-b \\
[Cx-d]_+ \end{matrix}\right\|_{\alpha} \quad \forall x \in
\rset^n,$$ where $\bar{x} = [x]_\mathcal{P} \equiv \arg\min_{z \in
\mathcal{P}} \norm{z-x}_{\beta}$. The constant $\theta(A,C)$ is the
Hoffman constant for the polyhedron $\mathcal{P}$ with respect to the pair
of norms $(\norm{\cdot}_{\alpha}, \norm{\cdot}_{\beta})$. In
\cite{KlaThi:95}, the authors provide  several estimates for the
Hoffman constant. Assume that $A$ has full row rank and define the
following quantity:
\begin{equation*}%\label{Hoff_fullrank}
\zeta_{\alpha,\beta} (A,C) := \min\limits_{I \in \mathcal{J}} \min_{u,v}
\left\{\norm{A^Tu + C^Tv}_{\beta^*} : \left\|\begin{matrix} u \\ v
\end{matrix}\right\|_{\alpha^*}=1, v_I \ge 0, v_{[m]\backslash
I}=0\right\},
\end{equation*}
where $\mathcal{J} = \{I \in 2^{[m]}: \text{card} \; I = r-p,
\text{rank}[A^T,\; C^T_I]=r \}$  and $r = \text{rank} [A^T, \; C^T]$. An alternative formulation of the
above quantity is:

\begin{equation}\label{Hoff_fullrank}
\frac{1}{\zeta_{\alpha,\beta} (A,C)} \!=\! \sup \left\{\!
\left\|\begin{matrix} u \\ v \end{matrix}\right\|_{\alpha^*}\!\!:
\begin{matrix}\norm{A^Tu + C^Tv}_{\beta^*}=1,
\text{rows of} \; C \\
 \text{corresponding to nonzero components of} \; v \\
\text{and rows of }\; A \; \text{are linearly independent}
\end{matrix} \!\right\}.
\end{equation}
In \cite{KlaThi:95} it was proved  that  $\zeta_{\alpha,\beta}(A,C)^{-1}$, where  $\zeta_{\alpha,\beta} (A,C)$ is defined in \eqref{Hoff_fullrank},  is the
Hoffman constant for the polyhedral set $\mathcal{P}$ w.r.t. the norms
$(\norm{\cdot}_{\alpha}, \norm{\cdot}_{\beta})$.

\noindent   Considering the Euclidean setting ($\alpha=\beta=2$) and
the above assumptions, then from previous discussion we have:
\[\theta (A,C) =  \max_{I \in \mathcal{J}}
\frac{1}{\sigma_{\min}([A^T,\; C^T_I]^T)}. \]
Under some regularity condition we can state a simpler
form for $\zeta_{\alpha,2}(A,C)$. Assume that $A$ has full row rank and that
the set $\{h \in \rset^n: \; Ah=0, \;  Ch<0 \}\neq \emptyset$, then,
we have \cite{KlaThi:95}:
\begin{equation}\label{Hoff_fullrank2}
\zeta_{\alpha,2} (A,C) := \min \left\{\norm{A^Tu + C^Tv}_2 :
\left\|\begin{matrix} u \\ v \end{matrix}\right\|_{\alpha^*}=1, v
\ge 0\right\}.
\end{equation}

\noindent Thus, for the special case  $m=0$,  i.e. there are no
inequalities, we have $\zeta_{2,2} (A,0) = \sigma_{\min}(A)$, where
$\sigma_{\min}(A)$ denotes the smallest nonzero singular value of
$A$,  and the Hofman constant is\footnote{This result can be also
proved using simple algebraic arguments. More precisely, from
Courant-Fischer theorem we know that $\|Ax\| \geq \sigma_{\min}(A)
\|x\|$ for all $x \in \text{Im}(A^T)$. Since we assume that our
polyhedron ${\cal P}=\{x: \; Ax=b\}$ is non-empty, then $x -
[x]_{\cal P} \in \text{Im}(A^T)$ for all $x \in \rset^n$ (from KKT
conditions of $\min_{z: Az=b} \|x - z\|^2$ we have that there exists
$\mu$ such that $x - [x]_{{\cal P}} + A^T \mu =0$). In conclusion,
we get:
\[ \| Ax - b\| = \| Ax - A [x]_{{\cal P}} \| \geq
\sigma_{\text{min}}(A) \|x - [x]_{{\cal P}}\| =
\sigma_{\text{min}}(A)  \text{dist}_2(x,\mathcal{P})  \quad \forall
x \in \rset^n. \]}:
\begin{equation}
\label{Hoff_sigma}
\theta (A,0) = \frac{1}{\sigma_{\min}(A)}.
\end{equation}

%%%%%%%%%%%%%%%%%%%%%%%%%%%%%%%%%%%%%%%%%%%%%%%%%%

\subsection{Composition of strongly convex function with
linear map is in $q\mathcal{S}_{L_f,\kappa_f}^{}(X)$} Let us
consider the class of optimization problems  (P) having the
following structured  form:
\begin{align}
\label{generalcls}
f^* =& \min_{x} f(x) \;  \equiv g(Ax)  \\
     & \text{s.t.}: \quad x \in X \equiv \{x \in \rset^n: \; Cx \leq d \}, \nonumber
\end{align}
i.e. the objective function is in the form $f(x) = g(Ax)$, where $g$
is a smooth and strongly convex  function and $A \in \rset^{m \times
n}$ is a nonzero general matrix. Problems  of this form arise in
various applications including dual formulations of linearly
constrained convex problems, convex quadratic  problems, routing
problems in data networks, statistical regression and many others.
Note that if $A$ has full column rank, then  $g(Ax)$ is strongly convex
function. However, if $A$ is rank deficient, then $g(Ax)$ is
not strongly convex. We prove in the  next theorem that
the objective function of problem \eqref{generalcls}
belongs to the class $q\mathcal{S}_{L_f,\kappa_f}^{}$.

\begin{theorem}
\label{the_gAx} Let $X =\{x \in \rset^n: \; Cx \leq d \}$ be a
polyhedral set, function $g:\rset^{m} \to \rset^{}$ be
$\sigma_g$-strongly convex with $L_g$-Lipschitz continuous
gradient on $X$, and $A \in \rset^{m \times n}$  be a nonzero matrix.
Then, the convex  function $f(x) = g(Ax)$  belongs to the class
$q\mathcal{S}_{L_f,\kappa_f}^{}(X)$, with constants $L_f = L_g
\|A\|^2$ and  $\kappa_f = \frac{\sigma_g}{\theta^2(A,C)}$, where
$\theta(A,C)$ is the Hoffman constant for the polyhedral  optimal
set~$X^*$.
\end{theorem}

\begin{proof}
The fact that $f$ has Lipschitz continuous gradient
follows immediately from the definition \eqref{lipg}. Indeed,
\begin{align*}
 \|\nabla f(x) - \nabla f(y) \| & =  \|A^T \nabla g(Ax) - A^T\nabla
g(Ay) \| \leq \|A\| \|\nabla g(Ax) - \nabla g(Ay) \| \\
& \leq \|A\| L_g \|Ax - Ay\| \leq \|A\|^2 L_g \| x - y\|.
\end{align*}
Thus, $L_f = L_g \|A\|^2$. Further, under  assumptions of the
theorem, there exists a unique pair $(t^*, T^*) \in \rset^{m} \times
\rset^{n}$ such that the following relations hold:
\begin{align}
\label{tstar}
A x^* = t^*, \quad  \nabla f(x^*) = T^* \quad \forall x^* \in X^*.
\end{align}
For completeness, we give a short proof of this well known fact
(see also \cite{LuoTse:92}):
let $x_1^*, x_2^*$ be two optimal points for the optimization
problem \eqref{generalcls}. Then, from convexity of $f$ and
definition of optimal points, it follows that:
\[  f \left( \frac{x_1^* + x_2^*}{2} \right) = \frac{f(x_1^*) + f(x_2^*)}{2}.  \]
Since $f(x) = g(Ax)$ we get from previous relation that:
\[ g \left( \frac{Ax_1^* + Ax_2^*}{2} \right) = \frac{g(Ax_1^*) + g(Ax_2^*)}{2}. \]
On the other hand using the definition of strong convexity
\eqref{sc} for $g$ we have:
\[  g \left( \frac{Ax_1^* + Ax_2^*}{2} \right) \leq  \frac{g(Ax_1^*) + g(Ax_2^*)}{2} -
 \frac{\sigma_g}{8} \|Ax_1^* - Ax_2^*\|^2. \]
Combining the previous two relations, we obtain that $Ax_1^* =
Ax_2^*$. Moreover, $\nabla f(x^*) = A^T \nabla g(A x^*)$.  In
conclusion,  $Ax$ and the gradient of $f$ are constant over the set
of optimal solutions $X^*$ for \eqref{generalcls}, i.e. the
relations \eqref{tstar} hold.  Moreover, we have that $f^* = f(x^*) = g(Ax^*) = g(t^*)$
for all $x^* \in X^*$. In conclusion, the set of optimal solutions
$X^*$ is described by the following polyhedral set:
\[ X^* =\{x^*: \;\; A x^* = t^*, \;\; C x^* \leq d  \}. \]
Since we assume that our optimization problem (P) has at least one solution, i.e. the optimal polyhedral set $X^*$ is non-empty, then from Hoffman inequality  we have that there exists some positive constant depending on the matrices $A$ and $C$
describing the polyhedral set $X^*$, i.e. $\theta(A,C)>0$, such that:
\[  \|x - \bar{x}\| \leq \theta(A,C) \left \| \begin{bmatrix}
     Ax - t^* \\
     [Cx - d]_+
\end{bmatrix} \right \| \quad \forall x \in \rset^n,  \]
where $\bar{x} = [x]_{X^*}$ (the projection of the vector $x$
onto the optimal set $X^*$).    Then, for any feasible $x$, i.e. $x$
satisfying $C x \leq d$, we have:
\[  \|x - \bar{x}\| \leq \theta(A,C) \|Ax - A \bar{x} \| \quad \forall x \in X.  \]
On the other hand, since $g$ is strongly convex,
it follows that:
\[ g(A \bar{x}) \overset{\eqref{sc_equiv}}{\geq} g(Ax) + \langle \nabla g(Ax), A\bar{x} - Ax \rangle +   \frac{\sigma_g}{2} \| Ax - A \bar{x}\|^2. \]
Combining the previous two relations and keeping in mind that $f(x) = g(Ax)$ and $\nabla f(x) =  A^T \nabla g(Ax)$,  we obtain:
\[ f^* \geq f(x) + \langle \nabla f(x), \bar{x} - x \rangle  + \frac{\sigma_g}{2 \theta^2(A,C)} \|x - \bar{x}\|^2 \quad \forall x \in X, \]
which proves that the quasi-strong convex inequality \eqref{wsc_basic} holds for the constant  $\kappa_f = \sigma_g/\theta^2(A,C)$.  \qed
\end{proof}

\noindent Note that we can relax the requirements for $g$ in Theorem
\ref{the_gAx}. For example, we can replace the  strong convexity assumption
on $g$ with the conditions that $g$ has unique minimizer $t^*$ and
it satisfies the quasi-strong convex condition \eqref{wsc_basic} with
constant $\kappa_g >0$. Then, using the same arguments as in the proof of Theorem
\ref{the_gAx}, we can show that for objective functions  $f(x) =
g(Ax)$ of problem (P),  the optimal set is $X^*=\{x^*: \;\; A x^* = t^*,
\;\; C x^* \leq d \} $ and $f$ satisfies   \eqref{wsc_basic} with constant $\kappa_f=
\frac{\kappa_g}{\theta^2(A,C)}$, provided that the corresponding
optimal set $X^*$ is nonempty.

\noindent Moreover, in the unconstrained case, that is $X = \rset^n$, and for objective
function $f(x) = g(Ax)$, we get from \eqref{Hoff_sigma} the following expression for the
quasi-strong convexity constant:
\begin{align}
\label{sigmamin} \kappa_f  = \sigma_g
\sigma_\text{min}^2(A).
\end{align}

\noindent Below we prove two extensions that belong to other functional
classes we have introduced in this paper.

%%%%%%%%%%%%%%%%%%%%%%%%%%%%%%%%%%%%%%%%%%%%%%%%%%%%%%%%%%%%%%%%%%

\subsection{Composition of strongly convex function with
linear map plus a linear term for $X = \rset^n$ is in
$\mathcal{G}_{L_f,\kappa_f}^{}(X)$}

\noindent Let us now consider  the class of unconstrained optimization problems  (P), i.e. $X = \rset^n$,  having the form:
\begin{equation}
\label{generalcls_c} f^*=  \min_{x \in \rset^n} \; f(x) \; \equiv
g(Ax) + c^T x,
\end{equation}
i.e. the objective function is in the form $f(x) = g(Ax) + c^Tx$,
where $g$ is a smooth and strongly convex  function,  $A \in
\rset^{m \times n}$ is a nonzero general matrix and $c \in \rset^n$.
We prove in the  next theorem that this type of objective function
for problem \eqref{generalcls_c} belongs to the class
$\mathcal{G}_{L_f,\kappa_f}^{}$:

\begin{theorem}
\label{the_gAx+c}
Under the same assumptions as in Theorem \ref{the_gAx} with $X =
\rset^n$, the  objective function of the form  $f(x) = g(Ax) + c^T
x$ belongs to the class $\mathcal{G}_{L_f,\kappa_f}^{}(X)$, with
constants $L_f = L_g \|A\|^2$ and  $\kappa_f =
\frac{\sigma_g}{\theta^2(A,0)}$, where $\theta(A,0)$ is the Hoffman
constant for the optimal set $X^*$.
\end{theorem}

\begin{proof}
Since  $g$ is $\sigma_g$-strongly convex and with $L_g$-Lipschitz
continuous gradient, then by the same reasoning as in the proof of
Theorem \ref{the_gAx} we get that there exists unique vector $t^*$
such that $A x^* = t^*$ for all $x^* \in X^*$. Similarly, there
exists unique scalar $s^*$ such that $c^T x^* = s^*$ for all $x^*
\in X^*$. Indeed, for $x_1^*, x_2^* \in X^*$ we have:
\[ f^* = g(t^*) + c^T x_1^* = g(t^*) + c^T x_2^*,   \]
which implies that $c^T x_1^* = c^T x_2^*$. On the other hand, since
problem (P) is unconstrained, for any $x^* \in X^*$ we have:
\[ 0 = \nabla f(x^*) = A^T \nabla g(t^*) + c,  \]
which implies that $c^T x^* = -(\nabla g(t^*))^T A x^* = - (\nabla
g(t^*))^T t^*. $ Therefore, the set of optimal solutions $X^*$ is
described in this case by the following polyhedron:
\[ X^* =\{x^*: \;\; A x^* = t^* \}. \]
Then, there exists $\theta(A,0)>0$ such that the Hoffman inequality holds:
\[  \|x - \bar{x}\| \leq \theta(A,0) \|Ax - A \bar{x} \|
 \quad \forall x \in \rset^n.  \]
From the previous inequality and  strong convexity of $g$, we have:
\begin{align*}
 \frac{\sigma_g}{\theta^2(A,0)} \|x - \bar{x}\|^2 & \leq  \sigma_g \|
Ax - A \bar{x} \|^2   \overset{\eqref{sc_equiv}}{\leq} \langle \nabla g(Ax) - \nabla
g(A\bar{x}), Ax - A \bar{x} \rangle \\
& =  \langle A^T \nabla g(Ax) + c - A^T \nabla g(A\bar{x}) - c, x -  \bar{x} \rangle \\
& = \langle \nabla f(x) - \nabla f(\bar{x}), x - \bar{x} \rangle.
\end{align*}
Finally, we conclude that the inequality on the variation of gradients \eqref{gvar} holds with constant  $\kappa_f  = \frac{\sigma_g}{\theta^2(A,0)}$.  \qed
\end{proof}

%%%%%%%%%%%%%%%%%%%%%%%%%%%%%%%%%%%%%%%%%%%%%%%%%%%%%%%%%%%%%%%%%%%%%%%%%%%5

\subsection{ Composition of strongly convex function with
linear map plus a linear term  is in
$\mathcal{F}_{L_f,\kappa_f}^{}(X_M)$}

\noindent Finally, let us now consider  the class of optimization problems  (P) of the form:
\begin{align}
\label{general_probl}
f^* = & \min_{x} f(x) \; \equiv g(Ax) + c^T x \\
      & \text{s.t.}: \quad x \in X \equiv \{x \in \rset^n: \; Cx \leq d \}, \nonumber
\end{align}
i.e. the objective function is in the form $f(x) = g(Ax) + c^Tx$,
where $g$ is a smooth and strongly convex  function,  $A \in
\rset^{m \times n}$ is a nonzero matrix and $c \in \rset^n$. We now
prove that  the objective function of problem \eqref{general_probl}
belongs to class $\mathcal{F}_{L_f,\kappa_f}^{}$, provided that some
boundedness assumption is imposed on~$f$.

\begin{theorem}
\label{the_gAx_cx} Under the same assumptions as in   Theorem
\ref{the_gAx}, the objective function $f(x) = g(Ax) + c^T x$ belongs
to the class $\mathcal{F}_{L_f,\kappa_f}^{}(X_M)$  for any constant
$M >0$ such that $X_M =\{x: \; x \in X, \;  f(x) - f^* \leq M \}$,
with constants $L_f = L_g \|A\|^2$ and  $\kappa_f =
\frac{\sigma_g}{\theta^2(A,c,C) \left(1 + M \sigma_g + 2 c_g^2
\right)}$, where $\theta(A,c,C)$ is the Hoffman constant for the
polyhedral optimal set $X^*$ and $c_g = \| \nabla  g(Ax^*) \|$, with
$x^* \in X^*$.
\end{theorem}

\begin{proof}
From  the proof of Theorem \ref{the_gAx+c} it follows that there exist unique $t^*$ and $s^*$ such that the optimal set of \eqref{general_probl} is given as follows:
\[ X^*  =\{x^*: \; A x^* = t^*, \;  c^T x^* = s^*, \; C x^* \leq d \}. \]
From Hoffman inequality we have that there exists some positive constant depending on the matrices $A, C$ and $c$ describing the polyhedral set $X^*$, i.e. $\theta(A,C, c)>0$, such that:
\[  \|x - \bar{x}\| \leq \theta(A,c,C) \left \| \begin{bmatrix}
     Ax - t^* \\
     c^T x - s^*\\
     [Cx - d]_+
\end{bmatrix} \right \| \quad \forall x \in \rset^n,  \]
where recall that $\bar{x} = [x]_{X^*}$. Then, for any feasible $x$, i.e.
satisfying $C x \leq d$, we have:
\begin{align}
\label{Hoffman_gAxcx}
\|x - \bar{x}\|^2 \leq \theta^2(A,c,C) \left( \|Ax - A \bar{x} \|^2 + (c^T x - c^T \bar{x})^2 \right) \quad \forall x \in X.
\end{align}
Since $f(x) = g(Ax) + c^T x$ and $g$ is strongly convex,
it follows from \eqref{sc_equiv} that:
\begin{align*}
g(Ax) - g(A \bar{x}) & \geq
\langle \nabla g(A \bar{x}), Ax - A \bar{x}  \rangle +  \frac{\sigma_g}{2} \| Ax - A \bar{x}\|^2 \\
& = \langle A^T \nabla g(A \bar{x}) + c, x -  \bar{x}  \rangle  - \langle c, x -  \bar{x} \rangle +  \frac{\sigma_g}{2} \| Ax - A \bar{x}\|^2 \\
& =\langle  \nabla f(\bar{x}), x -  \bar{x}  \rangle  - \langle c, x -  \bar{x} \rangle +  \frac{\sigma_g}{2} \| Ax - A \bar{x}\|^2.
\end{align*}
Using that $\langle  \nabla f(\bar{x}), x -  \bar{x}  \rangle \geq 0$ for all $x \in X$, and  definition of $f$, we obtain:
\begin{align}
\label{ineq1}
f(x)  - f^* \geq  \frac{\sigma_g}{2} \|A x - A \bar{x}\|^2  \quad \forall x \in X.
\end{align}

\noindent It remains to bound $(c^T x - c^T \bar{x})^2$. It is easy to notice that $\theta(A,c,C) \geq 1/\|c\|$.  We also observe that:
\[  c^T x - c^T \bar{x}  =  \langle \nabla f(\bar{x}), x -  \bar{x}  \rangle - \langle \nabla g(A \bar{x}), Ax - A \bar{x}  \rangle.  \]
Since $f(x) - f^* \geq \langle \nabla f(\bar{x}), x -  \bar{x}  \rangle \geq 0$ for all $x \in X$, then we obtain:
\[  |c^T x - c^T \bar{x}| \leq  f(x) - f^*  + \|\nabla g(A \bar{x})\| \;  \| Ax - A \bar{x} \|, \]
 and then using inequality $(\alpha+\beta)^2 \leq 2 \alpha^2 + 2 \beta^2$ and  considering $ f(x) - f^* \leq M$,  $c_g = \| \nabla  g(t^*) \|$ and \eqref{ineq1}, we get:
\begin{align*}
(c^T x - c^T \bar{x})^2 & \leq  2 (f(x) - f^*)^2  + 2 c_g^2 \| Ax - A \bar{x} \|^2 \\
& \leq  \left(2 M   + \frac{4 c_g^2}{\sigma_g} \right) \left( f(x) - f^* \right)  \quad \forall x \in X, \; f(x) - f^* \leq M.
\end{align*}
Finally, we conclude that:
\[ \| x  - \bar{x} \|^2   \leq  \frac{2 \theta^2(A,c,C)}{\sigma_g} \left( 1 + M \sigma_g + 2 c_g^2 \right) \left( f(x)  - f^* \right) \quad \forall x \in X, \; f(x) - f^* \!\leq\! M. \]
This proves the statement of the theorem.  \qed
\end{proof}

\noindent Typically, for feasible  descent methods we take $M =
f(x^0) - f^*$ in the previous theorem, where $x^0$ is the starting
point of the method. Moreover, if $X$ is bounded, then  there exists
always $M$  such that $f(x) - f^* \leq M$ for all $x \in X$. Note
that the requirement $f(x) - f^* \leq M$ for having a second order
growth inequality \eqref{wsc} for $f$ is necessary, as shown in the
following example:

\begin{example}
\label{example_S}
\noindent Let us consider  problem (P) in the form \eqref{general_probl} given by:
\[ \min_{x \in \rset^2_+} \frac{1}{2}x_1^2 + x_2  \]
which has $X^*=\{0\}$ and $f^* =0$. Clearly, there is no constant
$\kappa_f < \infty$ such that the following inequality to be valid:
\[  f(x) \geq \frac{\kappa_f}{2} \|x\|^2  \quad \forall x \geq 0.  \]
We can take for example $x_1 = 0$ and $x_2 \to + \infty$. However, for any
$M>0$ there exists $\kappa_f(M) < \infty$ satisfying the above
inequality for all $x \geq 0$ with  $f(x) \leq M$.  For example, we can take:
\[  \kappa_f(M) = \min\{1, \; \frac{1}{M} \}  \quad \Rightarrow \quad \mu_f(M) =\frac{1}{M} \quad \text{for} \quad M \geq 1. \]
Note that for this example $\theta(A,c,C) = \frac{1}{\|c\|} = 1$.  \qed
\end{example}

\noindent In the sequel we analyze  the convergence rate of  several  first order
methods for solving convex constrained optimization problem (P) having the objective function
in one of the functional classes introduced in this paper.

%%%%%%%%%%%%%%%%%%%%%%%%%%%%%%%%%%%%%%%%%%%%%%%%%%%%%%%%%%%%%%%%%%%%%
%%%%%%%%%%%%%%%%%%%%%%%%%%%%%%%%%%%%%%%%%%%%%%%%%%%%%%%%%%%%%%%%%%%%%55

\section{Linear convergence of first order methods}
\noindent We show in the next sections that a broad  class of  first
order methods, covering important particular algorithms, such as
projected gradient, fast gradient, random/cyclic coordinate descent,
extragradient descent and matrix splitting,   have linear
convergence rates on optimization problems (P), whose objective
function satisfies  one of the non-strongly convex  conditions given
above.

%%%%%%%%%%%%%%%%%%%%%%%%%%%%%%%%%%%%%%%%%%%%%%%%%%%\

\subsection{Projected gradient method (GM)}
In this section we consider the  projected  gradient algorithm with
variable step size:
\begin{center}
\framebox{
\parbox{6.1cm}{
\begin{center}
\textbf{ Algorithm (GM) }
\end{center}
{Given $x^0  \in X$   for $k \geq 1$ do:}
\[1.\;\; \text{Compute} \;\; x^{k+1} = \left [ x^k - \alpha_k \nabla f(x^k) \right ]_X \]  }}
\end{center}
where $\alpha_k$ is a step size such that $\alpha_k \in
[\bar{L}_f^{-1}, \; L_f^{-1}]$, with $\bar{L}_f \geq L_f$.

\subsubsection{Linear convergence of (GM) for $q\mathcal{S}_{L_f,\kappa_f}^{}$}
Let us show that the projected gradient method converges linearly on
optimization problems (P) whose objective functions belong to the
class $q\mathcal{S}_{L_f,\kappa_f}^{}$.

\begin{theorem}
\label{th_gm_W}  Let  the optimization  problem (P) have the
objective function  belonging to the class
$q\mathcal{S}_{L_f,\kappa_f}^{}$. Then, the sequence $x^k$ generated
by the projected  gradient method (GM) with constant step size
$\alpha_k = 1/L_f$  on (P) converges linearly to some optimal point
in $X^*$ with the rate:
\begin{align}
\label{lin_convW} \|x^{k} - \bar{x}^{k}\|^2  \leq \left( \frac{1 - \mu_f}{1 + \mu_f} \right)^{k} \|x^{0} - \bar{x}^{0}\|^2, \quad \text{where} \quad \mu_f = \frac{\kappa_f}{L_f}.
\end{align}
\end{theorem}

\begin{proof}
From Lipschitz continuity of the gradient   of $f$ given in
\eqref{lipg2} we have:
\begin{align}
\label{lip_pr} f(x^{k+1}) \leq f(x^k) + \langle \nabla f(x^k),
x^{k+1} - x^k\rangle + \frac{L_f}{2} \|x^{k+1} - x^k\|^2.
\end{align}
The optimality conditions for $x^{k+1}$ are:
\begin{align}
\label{opt_gm} \langle  x^{k+1} - x^k + \alpha_k \nabla f(x^k), x -
x^{k+1} \rangle \geq 0 \quad \forall x \in X.
\end{align}
Taking $x = x^k$ in  \eqref{opt_gm}  and replacing the corresponding
expression in \eqref{lip_pr}, we~get:
\[  f(x^{k+1}) \leq  f(x^k) + (\frac{L_f}{2} - \frac{1}{\alpha_k} ) \|x^{k+1} - x^k\|^2
 \overset{\alpha_k \leq L_f^{-1}}{\leq}  f(x^k) - \frac{L_f}{2} \|x^{k+1} - x^k\|^2. \]

\noindent Further, we have:
\begin{align*}
\|x^{k+1} - \bar{x}^k\|^2 & = \|x^{k} - \bar{x}^k\|^2   + 2 \langle x^k - \bar{x}^k,  x^{k+1} - x^k \rangle + \|x^{k+1} - x^k\|^2 \\
& = \|x^{k} - \bar{x}^k \|^2   + 2 \langle x^{k+1} - \bar{x}^k,  x^{k+1} - x^k \rangle - \|x^{k+1} - x^k\|^2 \\
& \overset{\eqref{opt_gm}}{\leq} \|x^{k} - \bar{x}^k\|^2   + 2 \alpha_k \langle \nabla f(x^{k}), \bar{x}^k - x^{k+1} \rangle - \|x^{k+1} - x^k\|^2\\
& = \|x^{k} \!- \bar{x}^k \|^2   \!+\! 2 \alpha_k \langle \nabla f(x^{k}), \bar{x}^k \!- x^k \rangle \!+\! 2 \alpha_k \langle \nabla f(x^{k}), x^k \!- x^{k+1} \rangle \\
& \qquad \qquad \qquad \qquad    - \|x^{k+1} - x^k\|^2 \\
& \overset{\eqref{wsc_basic}}{\leq} \|x^{k} - \bar{x}^k \|^2  + 2 \alpha_k  \Big( f^* - f(x^k) - \frac{\kappa_f}{2} \|x^k - \bar{x}^k\|^2 \Big) \\
& \quad - 2 \alpha_k \Big( \langle \nabla f(x^{k}),   x^{k+1} - x^k \rangle  + \frac{1}{2 \alpha_k} \|x^{k+1} - x^k\|^2 \Big) \\
& = (1 - \alpha_k \kappa_f) \|x^k - \bar{x}^k\|^2  +  2 \alpha_k  f^* \\
& \quad -  2 \alpha_k  \Big( f(x^k) + \langle \nabla f(x^{k}),   x^{k+1} - x^k \rangle  +  \frac{1}{2 \alpha_k}\|x^{k+1} - x^k\|^2 \Big) \\
& \overset{L_f \leq 1/\alpha_k}{\leq} (1 - \alpha_k \kappa_f) \|x^k - \bar{x}^k\|^2  +  2 \alpha_k  f^* \\
& \quad -  2 \alpha_k  \Big( f(x^k) + \langle \nabla f(x^{k}),   x^{k+1} - x^k \rangle  +  \frac{L_f}{2}\|x^{k+1} - x^k\|^2 \Big) \\
& \overset{\eqref{lip_pr}}{\leq} (1 - \alpha_k \kappa_f) \|x^{k} - \bar{x}^k\|^2 - 2 \alpha_k (f(x^{k+1}) - f^*).
\end{align*}
Since \eqref{wsc_basic} holds for the function $f$, then from Theorem \ref{th_qap_sog} we also have that
\eqref{wsc} holds  and therefore $f(x^{k+1}) - f^* \geq \frac{\kappa_f}{2} \| x^{k+1} -\bar{x}^{k+1}\|^2$. Combining the last inequality with the previous one and taking into account that $\|x^{k+1} - \bar{x}^{k+1}\| \leq \|x^{k+1} - \bar{x}^k\|$ , we get:
\[  \|x^{k+1} - \bar{x}^{k+1}\|^2 \leq  (1 - \alpha_k \kappa_f) \|x^{k} - \bar{x}^{k}\|^2 - \alpha_k \kappa_f
\|x^{k+1} - \bar{x}^{k+1}\|^2, \] or equivalently
\begin{align}
\label{lin_convw}
\|x^{k+1} - \bar{x}^{k+1}\|^2 \leq \frac{1 - \alpha_k \kappa_f}{1
+  \alpha_k \kappa_f} \cdot  \|x^{k} - \bar{x}^{k}\|^2.
\end{align}

\noindent However, the best  decrease is obtained for the constant step
size $\alpha_k = 1/L_f$ and using the definition of the condition number $\mu_f = \kappa_f/L_f$, we get:
\begin{align*}
 \|x^{k+1} - \bar{x}^{k+1}\|^2  \leq \frac{1 - \mu_f}{1 + \mu_f} \cdot \|x^{k} - \bar{x}^{k}\|^2.
\end{align*}
This proves our statement.   \qed
\end{proof}

\noindent Based on Theorem  \ref{th_gm_W} we can easily derive
linear convergence for the projected gradient algorithm (GM) in
terms of the function values:
\begin{align*}
f(x^{k+1}) & \overset{\eqref{lip_pr}}{\leq}  f(x^k) + \langle \nabla f(x^k),
x^{k+1} - x^k\rangle + \frac{L_f}{2} \|x^{k+1} - x^k\|^2 \\
& \overset{L_f \leq 1/\alpha_k}{\leq} \min_{x \in X} f(x^k) + \langle
\nabla f(x^k), x - x^k\rangle + \frac{1}{2\alpha_k} \|x^k - x\|^2 \\
& \leq  \min_{x \in X} f(x) + \frac{1}{2\alpha_k} \|x^k - x\|^2  \leq f(\bar{x}^k)  + \frac{1}{2\alpha_k} \|x^k - \bar{x}^k\|^2 \\
& \overset{\eqref{lin_convW}}{\leq} f^* +  \frac{\bar{L}_f}{2}
\left( \frac{1 - \mu_f}{1 + \mu_f} \right)^{k} \|x^{0} - \bar{x}^{0}\|^2.
\end{align*}
Finally, the best convergence rate is obtained for constant step
size $\alpha_k = 1/L_f$:
\begin{align}
\label{lin_conv_W_uc} f(x^k)  - f^* \overset{\eqref{lin_convS}}{\leq}
\frac{L_f \|x^{0} - \bar{x}^{0}\|^2}{2} \left( \frac{1 - \mu_f}{1 + \mu_f} \right)^{k-1} \qquad \forall k \geq 1.
\end{align}
However, this rate is not continuous as $\mu_f \to 0$. For simplicity, let us assume  constant step size $\alpha_k = 1/L_f$,  and then, using that (GM) is a descent method, i.e.  $f(x^{k}) - f^* \leq  f(x^{k-j}) - f^*$ for all $j < k$ and iterating the main inequality from the proof of Theorem \ref{th_gm_W},  we obtain:
\begin{align*}
\|x^k - \bar{x}^k\|^2  &\leq (1 - \mu_f) \|x^{k-1} - \bar{x}^{k-1}\|^2 - \frac{2}{L_f} \left( f(x^k) - f^* \right ) \\
%& \leq (1 - \mu_f)^2 \|x^{k-2} - \bar{x}^{k-2}\|^2 - \frac{2}{L_f} \left( (1-\mu_f) %(f(x^{k-1}) - f^*) +  (f(x^k) - f^*) \right )\\
& \leq  (1 - \mu_f)^k \|x^{0} - \bar{x}^{0}\|^2 - \frac{2}{L_f} \sum_{j=0}^k (1-\mu_f)^j (f(x^{k-j}) - f^*) \\
& \leq (1 - \mu_f)^k \|x^{0} - \bar{x}^{0}\|^2  - \frac{2}{L_f} \left( f(x^{k}) - f^* \right) \sum_{j=0}^k (1-\mu_f)^j.
\end{align*}
Finally,  we get linear convergence in terms of the
function values:
\begin{align}
\label{lin_conv_fW} f(x^k)  - f^* \leq
\frac{  L_f \|x^{0} - \bar{x}^{0}\|^2}{2} \cdot \frac{\mu_f}{(1-\mu_f)^{-k} - 1}.
\end{align}
Since $(1 + \alpha)^k \to 1 + \alpha k$ as $\alpha \to 0$, then we see that:
\[  \frac{\mu_f}{(1-\mu_f)^{-k} - 1} \leq  \frac{1}{k}  \quad \text{as} \quad \mu_f \to 0, \]
and thus from  \eqref{lin_conv_fW} we recover the classical
sublinear rate for (GM) as $\mu_f \to 0$:
\[ f(x^k)  - f^* \leq
\frac{  L_f \|x^{0} - \bar{x}^{0}\|^2}{2k}  \quad \text{as} \quad
\mu_f \to 0. \]

%In conclusion, we can interpolate between the right
%hand side terms in \eqref{lin_conv_W_uc} and \eqref{lin_conv_fW} to
%obtain convergence rates in terms of function values of the form:
%\begin{align*}
%f(x^k)  - f^* & \leq  \frac{  L_f \|x^{t} - \bar{x}^{t}\|^2}{2}
%\frac{\mu_f}{(1-\mu_f)^{t-k} - 1} \\
%& \leq  \frac{  L_f \|x^{0} - \bar{x}^{0}\|^2}{2}
%\frac{\mu_f}{(1-\mu_f)^{t-k} - 1}  \left( \frac{1 - \mu_f}{1 +
%\mu_f} \right)^{t} \quad \forall t=0:k-1,
%\end{align*}
%or equivalently
%\begin{align*}
%f(x^k)  - f^* & \leq   \frac{  L_f \|x^{0} - \bar{x}^{0}\|^2}{2}
% \min_{t=0:k-1} \frac{\mu_f}{(1-\mu_f)^{t-k} - 1}  \left( \frac{1 - \mu_f}{1 +
%\mu_f} \right)^{t}.
%\end{align*}

%If we defined $R_f=\|x^0 - \bar{x}^0\|$, then we obtain an
%$\epsilon$-approximate solution in:  \[ \mathcal{O} \left(
%\frac{1}{\mu_f^s} \log \frac{L_f R^2_f}{\epsilon} \right) \quad \text{iterations}. \]

%%%%%%%%%%%%%%%%%%%%%%%%%%%%%%%%%%%%%%%%%%%%%%%%%%%%%%%%%%%%%%%%%%%%%%%

\subsubsection{Linear convergence of (GM) for $\mathcal{F}_{L_f,\kappa_f}^{}$}
We now show that the projected gradient method converges linearly on
optimization problems (P) whose objective functions belong to the
class $\mathcal{F}_{L_f,\kappa_f}^{}$.

\begin{theorem}
\label{th_gm_S} Let  optimization  problem (P) have objective
function  belonging to the class $\mathcal{F}_{L_f,\kappa_f}^{}$.
Then, the sequence $x^k$ generated by the projected  gradient method
(GM) with constant step size $\alpha_k = 1/L_f$ on (P) converges
linearly to some optimal point in $X^*$ with the rate:
\begin{align}
\label{lin_convS} \|x^{k} - \bar{x}^{k}\|^2  \leq \left( \frac{1}{1 + \mu_f} \right)^{k} \|x^{0} - \bar{x}^{0}\|^2, \quad \text{where} \quad \mu_f = \frac{\kappa_f}{L_f}.
\end{align}
\end{theorem}

\begin{proof}
Using similar arguments as in the previous Theorem  \ref{th_gm_W}, we have:
\begin{align*}
\|x^{k+1} - x\|^2 & = \|x^{k} - x\|^2   + 2 \langle x^k - x,  x^{k+1} - x^k \rangle + \|x^{k+1} - x^k\|^2 \\
& = \|x^{k} - x\|^2   + 2 \langle x^{k+1} - x,  x^{k+1} - x^k \rangle - \|x^{k+1} - x^k\|^2 \\
& \overset{\eqref{opt_gm}}{\leq} \|x^{k} - x\|^2   + 2 \alpha_k \langle \nabla f(x^{k}),
x - x^{k+1} \rangle - \|x^{k+1} - x^k\|^2\\
& \leq \|x^{k} - x\|^2  - 2 \alpha_k  \Big(  \langle \nabla
f(x^{k}), x^{k+1} - x \rangle + \frac{L_f}{2} \|x^{k+1} - x^k\|^2 \\
& \qquad \qquad \qquad \qquad \qquad \qquad \qquad \quad + (\frac{1}{2\alpha_k} - \frac{L_f}{2}) \|x^{k+1} - x^k\|^2 \Big )\\
& =  \|x^{k} - x\|^2 + (L_f \alpha_k -1) \|x^{k+1} - x^k\|^2 \\
&  \; -2\alpha_k \!\left(\! \langle \nabla f(x^{k}), x_k \!-\! x
\rangle \!+\! \langle \nabla f(x^{k}), x^{k+1} \!-\! x_k \rangle
\!+\! \frac{L_f}{2}
\|x^{k+1} \!-\! x^k\|^2 \!\right)\\
&\overset{\eqref{lip_pr}}{\leq}
\|x^{k} - x\|^2
 + (L_f \alpha_k -1) \|x^{k+1} - x^k\|^2 \\
& \qquad \qquad   +  2 \alpha_k (f(x) - f(x^k))  + 2 \alpha_k
(f(x^k) -
f(x^{k+1})) \\
& \overset{\alpha_k \leq L_f^{-1}}{\leq}  \|x^{k} -
x\|^2   - 2 \alpha_k (f(x^{k+1}) - f(x)) \qquad \forall x \in X.
\end{align*}
Taking now in the previous relations $x =\bar{x}^{k}$, using
$\|x^{k+1} - \bar{x}^{k+1}\| \leq \|x^{k+1} - \bar{x}^k\|$ and the quadratic functional
growth  of $f$  \eqref{wsc}, we get:
\[  \|x^{k+1} - \bar{x}^{k+1}\|^2 \overset{\eqref{wsc}}{\leq} \|x^{k} - \bar{x}^{k}\|^2 - \kappa_f  \alpha_k  \|x^{k+1} - \bar{x}^{k+1}\|^2 \]
or equivalently
\begin{align}
\label{lin_convS1}
\|x^{k+1} - \bar{x}^{k+1}\|^2 \leq \frac{1}{1
+  \kappa_f \alpha_k} \|x^{k} - \bar{x}^{k}\|^2.
\end{align}

\noindent However, the best  decrease is obtained for the constant step
size $\alpha_k = 1/L_f$ and using the definition of the condition number $\mu_f = \kappa_f/L_f$, we get:
\begin{align*}
 \|x^{k+1} - \bar{x}^{k+1}\|^2  \leq \frac{1}{1 +   \mu_f} \|x^{k} - \bar{x}^{k}\|^2.
\end{align*}
\noindent Thus,  we  have obtained the  linear convergence rate  for
(GM) with constant step size $\alpha_k = 1/L_f$ from the theorem.
\qed
\end{proof}

\noindent Using similar arguments as for \eqref{lin_conv_W_uc} and
combining with  \eqref{lin_convS1} we can also derive linear
convergence of (GM) in terms of the function values:
\begin{align*}
f(x^{k+1}) - f^* & \leq \frac{1}{2\alpha_k} \|x^{k} -
\bar{x}^{k}\|^2 \overset{\eqref{lin_convS1}}{\leq}
\frac{1}{2\alpha_k} \left( \frac{1}{1 +  \kappa_f \alpha_k}
\right) \|x^{k-1} - \bar{x}^{k-1}\|^2,
\end{align*}
and the best convergence rate is obtained for constant step size
$\alpha_k = 1/L_f$:
\begin{align}
\label{lin_conv_f} f(x^k)  - f^* \overset{\eqref{lin_convS}}{\leq}
\frac{L_f \|x^{0} - \bar{x}^{0}\|^2}{2} \left( \frac{1}{1 + \mu_f} \right)^{k-1} \qquad \forall k \geq 1.
\end{align}
However, this rate is not continuous as $ \mu_f \to 0$.   We can
interpolate between the right hand side terms in \eqref{lin_conv_Lf}
and \eqref{lin_conv_f} to obtain convergence rates in terms of
function values of the form:
\begin{align*}
f(x^k)  - f^* & \leq  \frac{  L_f \|x^{t} - \bar{x}^{t}\|^2}{2(k-t)}
 \leq  \frac{  L_f \|x^{0} - \bar{x}^{0}\|^2}{2(k-t)}
\frac{1}{(1+\mu_f)^{t}}  \quad \forall t=0:k-1,
\end{align*}
or equivalently
\begin{align*}
f(x^k)  - f^* & \leq   \frac{  L_f \|x^{0} - \bar{x}^{0}\|^2}{2}
 \min_{t=0:k-1} \frac{1}{(1+\mu_f)^{t}(k-t)}.
\end{align*}

\noindent Finally, in the next theorem we establish necessary and
sufficient conditions  for linear convergence of the gradient method
(GM).
\begin{theorem}
On the class of optimization problems (P) the sequence generated by
the  gradient method (GM) with constant step size  is converging
linearly to some optimal point in $X^*$  if and only if the
objective function $f$ satisfies the quadratic functional growth \eqref{wsc},
i.e $f$ belongs to the functional  class $\mathcal{F}_{L_f,\kappa_f}^{}$.
\end{theorem}
\begin{proof}
The fact that linear convergence of the gradient method implies $f$
satisfying the second order growth property  \eqref{wsc} follows
from  Theorem \ref{th_necesary_linconv}. The other implication
follows from  Theorem \ref{th_gm_S}, eq.  \eqref{lin_convS1}. \qed
\end{proof}

%If we defined $R_f=\|x^0 - \bar{x}^0\|$, then we obtain an
%$\epsilon$-approximate solution in:
%\[ \mathcal{O} \left( \frac{1}{\mu_f} \log \frac{L_f R^2_f}{\epsilon} \right) \quad \text{iterations}.\]

%%%%%%%%%%%%%%%%%%%%%%%%%%%%%%%%%%%%%%%%%%%%%%%%%%%%%%%%%%%%%%%%%%%%%%%%%%%
%%%%%%%%%%%%%%%%%%%%%%%%%%%%%%%%%%%%%%%%%%%%%%%%%%%%%%%%%%%%%%%%%%%%%%%%%%%%

\subsection{Fast gradient method (FGM)}
In this section we consider the following fast gradient algorithm,
which is a version of Nesterov's optimal gradient method
\cite{Nes:04}:
\begin{center}
\framebox{
\parbox{9.1cm}{
\begin{center}
\textbf{ Algorithm {\bf (FGM)} }
\end{center}
{Given $x^0 = y^0 \in X$, for $k\geq 1$ do:}
\begin{enumerate}
\item Compute ${x}^{k+1}=\left[y^k -  \frac{1}{L_f} \nabla f(y^k)\right]_X$ and
\item $y^{k+1} = x^{k+1} + \beta_k \left(x^{k+1} - x^{k}\right)$
\end{enumerate}
}}
\end{center}
for appropriate choice of the parameter $\beta_k >0$ for all $k \geq 0$.

\subsubsection{Linear convergence of (FGM)  for $q\mathcal{S}_{L_f,\kappa_f}^{}$.}
When the objective function $f \in q\mathcal{S}_{L_f,\kappa_f}^{}(X)$
we take the following expression for the parameter $\beta_k$:
\[  \beta_k = \frac{\sqrt{L_f} - \sqrt{\kappa_f}}{ \sqrt{L_f} + \sqrt{\kappa_f}} \quad \forall k \geq 0. \]

\noindent First of all we can easily observe that if $f \in
q\mathcal{S}_{L_f,\kappa_f}^{}(X)$, then the gradient mapping $g(x)$
satisfies the following inequality:
\begin{align}
\label{prop_gmap1}
f^* \geq f(x^+) + \langle g(x), \bar{x} -x \rangle + \frac{1}{2L_f} \|g(x)\|^2 + \frac{\kappa_f}{2}
 \| \bar{x} - x\| \equiv  q_{L_f,\kappa_f}(\bar{x},x)
\end{align}
for all $x \in \rset^n$ (recall that $\bar{x} = [x]_{X^*}$  and $x^+
= [x - 1/L_f \nabla f(x)]_X$). The convergence proof follows similar
steps as in  \cite{Nes:04}[Section 2.2.4].

\begin{lemma}
\label{lemma_fg1} Let  optimization  problem (P) have the objective
function $f$ belonging to the class $q\mathcal{S}_{L_f,\kappa_f}^{}$
and an arbitrary sequence $\{y^k\}_{k\geq 0}$  satisfying $\bar{y}^k
= [y^k]_{X^*} = y^*$ for all $k \geq 0$. Define an initial function:
\[ \phi_0(x) = \phi_0^* + \frac{\gamma_0}{2}\| x - v^0 \|^2, \quad \text{where} \;\; \gamma_0= \kappa_f, \;\; v^0=y^0 \; \text{and} \;\;  \phi_0^* = f(y^0), \]
and a sequence $\{\alpha_k\}_{k\geq 0}$ satisfying $\alpha_k \in (0,\; 1)$. Then, the following two sequences, iteratively defined as:
\begin{align}
\label{construction_fg}
& \lambda_{k+1} = (1 - \alpha_k) \lambda_k, \quad \text{with} \quad \lambda_0=1, \nonumber \\
&\phi_{k+1}(x) = (1 - \alpha_k) \phi_k(x)  \\
& \qquad \qquad \quad + \alpha_k \!
\left(\! f(x^{k+1}\!)  \!+\! \frac{1}{2L_f} \|g(y^k)\|^2  \!+\!
\langle g(y^k), x \!-\! y^k \rangle \!+\! \frac{\kappa_f}{2} \|x
\!-\! y^k \|^2 \!\right)\!, \nonumber
\end{align}
where $x^0 = y^0$ and ${x}^{k+1}=\left[y^k -  \frac{1}{L_f} \nabla f(y^k)\right]_X$, satisfy the following property:
\begin{align}
\label{prop1_fg}
\phi_k(y^*) \leq (1 - \lambda_k) f^* + \lambda_k \phi_0(y^*) \quad \forall k \geq 0.
\end{align}
\end{lemma}

\begin{proof}
We prove this statement  by induction. Since $\lambda_0=1$, we observe that:
\[  \phi_0(y^*) = (1 - \lambda_0) f^* + \lambda_0 \phi_0(y^*).  \]
Assume that  the following inequality is valid:
\begin{align}
\label{induction_fg}
\phi_k(y^*) \leq (1 - \lambda_k) f^* + \lambda_k \phi_0(y^*),
\end{align}
then we have:
\begin{align*}
\phi_{k+1}(y^*) & = \phi_{k+1}(\bar{y}^k) = (1 - \alpha_k) \phi_k(\bar{y}^k) + \alpha_k q_{L_f,\kappa_f}(\bar{y}^k, y^k) \\
& \overset{\eqref{prop_gmap1}}{\leq} (1 - \alpha_k) \phi_k(\bar{y}^k) + \alpha_k f^* \\
& = [1 - (1 - \alpha_k)\lambda_k] f^* + (1 - \alpha_k) \left( \phi_k(\bar{y}^k)  - (1 - \lambda_k) f^* \right) \\
& \overset{ \bar{y}^k = y^* + \eqref{induction_fg} }{\leq} (1 -\lambda_{k+1}) f^* + \lambda_{k+1} \phi_0(y^*).
\end{align*}
which proves our statement. \qed
\end{proof}

\begin{lemma}
\label{lemma_fg2} Under the same assumptions as in Lemma
\ref{lemma_fg1} and assuming also that the sequence $\{x_k\}_{k \geq
0}$, defined as $x^0=y^0$  and  ${x}^{k+1}=\left[y^k - \frac{1}{L_f}
\nabla f(y^k)\right]_X$,  satisfies:
\begin{align}
\label{prop2_fg}
f(x^k) \leq \phi_k^* = \min_{x \in \rset^n} \phi_k(x) \qquad \forall k \geq 0,
\end{align}
then we obtain the following convergence:
\begin{align}
\label{prop3_fg}
f(x^k)- f^* \leq \lambda_k  \left( f(x^0) - f^* + \frac{\gamma_0}{2}\| y^* - y^0 \| \right).
\end{align}
\end{lemma}

\begin{proof}
Indeed we have:
\begin{align*}
f(x^k) - f^* & \leq \phi_k^* - f^*= \min_{x \in \rset^n} \phi_k(x) - f^*  \leq \phi_k(y^*) - f^*\\
& \overset{\eqref{prop1_fg}}{\leq} (1 - \lambda_k) f^* + \lambda_k \phi_0(y^*) - f^* = \lambda_k \left(  \phi_0(y^*) - f^* \right),
\end{align*}
which proves the statement of the lemma. \qed
\end{proof}

\begin{theorem}
\label{th_fgm_W} Under the same assumptions as in Lemma
\ref{lemma_fg1}, the sequence $x^k$ generated by  fast gradient
method (FGM) with constant parameter $\beta_k = (\sqrt{L_f} -
\sqrt{\kappa_f})/(\sqrt{L_f} + \sqrt{\kappa_f})$  converges linearly
in terms of function values with the rate:
\begin{align}
\label{lin_conv_fgmW} f(x^{k}) - f^*  \leq \left( 1 - \sqrt{\mu_f} \right)^{k} \cdot 2 \left(f(x^0) - f^* \right), \quad \text{where} \quad \mu_f = \frac{\kappa_f}{L_f},
\end{align}
provided that all  iterates $y^k$ produce the same
projection\footnote{See Remark \ref{remark_all_bary} below for an
example satisfying this condition.} onto  optimal set~$X^*$.
\end{theorem}

\begin{proof}
Let us consider $x^0 = y^0 = v^0 \in X$. Further,   for the sequence
of functions $\phi_k(x)$ as defined in \eqref{construction_fg}  take
$\alpha_k = \sqrt{\mu_f} \in (0,\; 1)$ for all $k\geq 0$ and denote
$\alpha = \sqrt{\mu_f}$. First, we need to show that the method
(FGM) defined  above generates a sequence $x^k$ satisfying $\phi_k^*
\geq f(x^k)$. Assuming that $\phi_k(x)$ has the following two
properties:
\[  \phi_k(x) = \phi_k^* + \frac{\kappa_f}{2}\| x - v^k \|^2 \quad \text{and} \quad \phi_k^* \geq f(x^k),   \]
where $\phi_k^* = \min_{x \in \rset^n} \phi_{k}(x)$  and $v^k= \arg \min_{x \in \rset^n} \phi_k(x)$, then we will show that $\phi_{k+1}(x)$  has similar properties. First of all, from the definition of $\phi_{k+1}(x)$, we get:
\[  \nabla^2 \phi_{k+1}(x) = \left((1 -\alpha) \kappa_f + \alpha \kappa_f \right) I_n = \kappa_f I_n, \]
i.e. $\phi_{k+1}(x)$ is also a quadratic function of the same form as $\phi_{k}(x)$:
\[ \phi_{k+1}(x) = \phi_{k+1}^* + \frac{\kappa_f}{2}\| x - v^{k+1} \|^2, \]
where  the expression of $v^{k+1}= \arg \min_{x \in \rset^n} \phi_{k+1}(x)$ is obtained from the equation $\nabla \phi_{k+1}(x) =0$, which leads to:
\[  v^{k+1} = \frac{1}{\kappa_f} \left (  (1-\alpha) \kappa_f v^k + \alpha \kappa_f y^k - \alpha g(y^k) \right).   \]
Evaluating $\phi_{k+1}$ in $y^k$ leads to:
\begin{align*}
\phi_{k+1}^*  + \frac{\kappa_f}{2}\| y^k - v^{k+1} \|^2 =&   (1 -\alpha) \left( \phi_k^* + \frac{\kappa_f}{2}\| y^k - v^k \|^2 \right) \\
& + \alpha  \left(\!  f(x^{k+1}) +  \frac{1}{2L_f} \|g(y^k)\|^2  \right).
\end{align*}
On the other hand, we have:
\[ v^{k+1} - y^k = \frac{1}{\kappa_f} \left(  \kappa_f(1-\alpha)(v^k - y^k) -
 \alpha g(y^k) \right). \]
If we substitute this expression above, we obtain:
\begin{align*}
\phi_{k+1}^* =&  (1 -\alpha) \phi_k^* + \alpha f(x^{k+1}) +  \left( \frac{\alpha}{2L_f} -  \frac{\alpha^2}{2\kappa_f} \right) \|g(y^k)\|^2 \\
& +   \alpha (1 - \alpha)\left(
\frac{\kappa_f}{2} \| y^k - v^k \|^2 + \langle g(y^k), v^k - y^k \rangle \right).
\end{align*}
Using the main property of the gradient mapping \eqref{gmap}, valid for functions with Lipschitz continuous gradient,  we have:
\[  \phi_k^* \geq f(x^k) \geq f(x^{k+1}) +  \langle g(y^k), x^k - y^k \rangle + \frac{1}{2L_f} \|g(y^k)\|^2. \]
Substituting this inequality in the previous one we get:
\[  \phi_{k+1}^* \geq f(x^{k+1}) + \left( \frac{1}{2L_f} -  \frac{\alpha^2}{2\kappa_f} \right) \|g(y^k)\|^2 + (1 -\alpha) \langle g(y^k), \alpha(v^k-y^k) + x^k - y^k \rangle. \]
Since $\alpha = \sqrt{\mu_f}$, then $\frac{1}{2L_f} -  \frac{\alpha^2}{2  \kappa_f} =0$. Moreover, we have the freedom to choose $y^k$, which is obtained from the condition $\alpha(v^k-y^k) + x^k - y^k =0$:
\[  y^k = \frac{1}{1 + \alpha} ( \alpha v^k + x^k).  \]
Then, we can conclude that $\phi_{k+1}^* \geq f(x^{k+1})$. Moreover, replacing the expression of $y^k$ in $v^{k+1}$ leads to the conclusion that we can eliminate the sequence $v^k$ since it can be expressed as: $v^{k+1} = x^k + \frac{1}{\alpha} (x^{k+1} - x^k)$. Then, we find that $y^{k+1}$ has the expression as in our scheme (FGM) above with $\beta_k= (\sqrt{L_f} - \sqrt{\kappa_f})/(\sqrt{L_f} + \sqrt{\kappa_f})$. Using, now Lemmas  \ref{lemma_fg1} and \ref{lemma_fg2} we get the convergence rate from \eqref{lin_conv_fgmW} (we also use that $\frac{\kappa_f}{2} \|x^{0} - \bar{x}^{0}\|^2 \leq f(x^0) - f^*$). \qed
\end{proof}

\begin{remark}
\label{remark_all_bary}
For unconstrained problem $\min_{x \in \rset^n} g(Ax)$, the gradient in some point $y$ is given by $A^T \nabla g(Ay) \in \text{Range}(A^T)$. Then,
the method (FGM) generates in this case a sequence $y^k$ of the form:
\[ y^k = y^0 + A^T z^k, \qquad z^k \in \rset^m \quad  \forall k \geq 0.   \]
Moreover, for this problem the optimal set $X^*=\{x: \; Ax = t^*\}$ and the projection onto this affine subspace is given by:
\[  [ \ \cdot \ ]_{X^*} = \left(I_n - A^T(A A^T)^{-1} A \right) (\cdot) + A^T(AA^T)^{-1} t^*.  \]
In conclusion, all vectors $y^k$ generated by algorithm (FGM) produce the same projection onto the optimal set $X^*$:
\[ \bar{y}^k = y^0  - A^T(AA^T)^{-1}Ay^0 + A^T(AA^T)^{-1} t^* \quad \forall k \geq 0,
\] i.e. the assumptions of Theorem  \ref{th_fgm_W}  are valid for this optimization problem. \qed
\end{remark}

%%%%%%%%%%%%%%%%%%%%%%%%%%%%%%%%%%%%%%%%%%%%%%%%%%%%%%%%%%%%%%%%%%%%%%%%%%%%%%%%%

\subsubsection{Linear convergence of restart (FGM)
 for $\mathcal{F}_{L_f,\kappa_f}^{}$.}
It is known  that for the convex optimization problem (P), whose
objective function  $f$ has Lipschitz continuous  gradient, and for the
choice:
\[ \beta_k = \frac{\theta_k -1 }{\theta_{k+1}}, \quad \text{with} \quad  \theta_1=1 \; \text{and} \;  \theta_{k+1} = \frac{1 + \sqrt{1 + 4\theta_k^2 }}{2},  \]
the algorithm {(FGM)} has the following convergence rate
\cite{Nes:04,DonCan:13}:
\begin{equation}
\label{bound_dual_optim_dfg}
 f({x}^{k}) - f^* \leq \frac{2 L_f \|x^0 - \bar{x}^0\|^2}{(k+1)^2} \qquad  \forall k >0.
\end{equation}
We will show next that  on the optimization problem (P) whose
objective function satisfies additionally the quadratic functional growth
\eqref{wsc}, i.e.  $f \in
\mathcal{F}_{L_f,\kappa_f}^{}$,  a  restarting version of
algorithm {(FGM)} with the above choice of $\beta_k$ has linear
convergence without the  assumption $\bar{y}^k = y^*$ for all $k \geq 0$.
Restarting variants of (FGM) have been also considered
in other contexts, see e.g. \cite{DonCan:13}.  By fixing a positive
constant $c \in (0, \; 1)$ and then combining
\eqref{bound_dual_optim_dfg} and \eqref{wsc}, we get:
\[ f(x^k) - f^*  \leq \frac{2 L_f }{(k+1)^2} \| x^0 - \bar{x}^0\|^2 \leq
 \frac{4 L_f }{\kappa_f (k+1)^2} (f(x^0) - f^*) \leq c (f(x^0) - f^*), \]
which leads to the following expression:
\[  c = \frac{4 L_f }{\kappa_f k^2}. \]
Then, for fixed $c$, the number of iterations $K_c$ that we need to
perform in order to obtain $f(x^{K_c}) - f^* \leq c (f(x^0) - f^*) $
is given by:
\[ K_c = \left  \lceil \sqrt{\frac{4 L_f }{c \kappa_f}} \; \right \rceil = \left  \lceil \sqrt{\frac{4}{c \mu_f}} \; \right \rceil.  \]
Therefore, after each $K_c$ steps of Algorithm {(FGM)} we
restart it obtaining the following scheme:
\begin{center}
\framebox{
\parbox{10cm}{
\begin{center}
\textbf{ Algorithm {\bf (R-FGM)} }
\end{center}
{Given $x^{0,0} = y^{0,0} = x^0 \in X$ and restart interval $K_c$. For $j
\geq 0$ do:}
\begin{enumerate}
\item Run Algorithm (FGM) for $K_c$ iterations to get ${x}^{K_c,j}$
\item Restart: $x^{0,j+1} = x^{K_c,j}$, \; $y^{0,j+1} = x^{K_c,j}$ \; and \;
$\theta_1=1$.
\end{enumerate}
}}
\end{center}
Then, after $p$ restarts of Algorithm {(R-FGM)} we obtain the linear
convergence:
\begin{align*}
f(x^{0,p}) - f^* & = f(x^{K_c,p-1}) - f^* \leq  \frac{2 L_f \|
x^{0,p-1} - \bar{x}^{0,p-1}\|^2}{(K_c+1)^2} \\
& \leq c (f(x^{0,p-1}) - f^*) \leq \cdots \leq c^p (f(x^{0,0}) -
f^*)  =  c^p (f(x^{0}) - f^*).
\end{align*}
Thus, total number of iterations is $k = p K_c$ and denote $x^k = x^{0,p}$. Then, we have:
\[ f(x^k) - f^* \leq \left(  c^{\frac{1}{K_c}} \right)^k  (f(x^{0}) - f^*).   \]
We want to optimize e.g. the number of iteration $K_c$:
\[ \min_{K_c} c^{\frac{1}{K_c}} \quad \Leftrightarrow  \quad \min_{K_c} \frac{1}{K_c} \log c
\quad \Leftrightarrow  \quad \min_{K_c} \frac{1}{K_c} \log \frac{4}{\mu_f K_c^2}, \]
which leads to
\[ K_c^* = \frac{2 e}{\sqrt{\mu_f}} \quad \text{and} \quad c = e^{-2}.  \]
In conclusion, we get the following convergence rate for (R-FGM) method:
\begin{align}
\label{conv_r-fg}
f(x^k) - f^* \leq  \left(  e^{-2 \frac{\sqrt{\mu_f}}{2e}} \right)^k  (f(x^{0}) - f^*) =
\left(  e^{-\frac{\sqrt{\mu_f}}{e}} \right)^k (f(x^{0}) - f^*),
\end{align}
and since $e^\alpha \approx 1+ \alpha$ as $\alpha \approx 0$, then for  $\frac{\sqrt{\mu_f}}{e} \approx 0$ we get:
\begin{align}
\label{conv_r-fg1}
f(x^k) - f^* \leq   \left(  e^{-\frac{\sqrt{\mu_f}}{e}} \right)^k (f(x^{0}) - f^*) \approx  \left(1 - \frac{\sqrt{\mu_f}}{e} \right)^k (f(x^{0}) - f^*).
\end{align}
Note that if the optimal value $f^*$ is known in advance, then we just need to
restart algorithm {(R-FGM)} at the iteration $\bar{K}_c \leq
K_c^*$ when the following condition holds:
\[   f(x^{\bar{K}_c,j}) - f^* \leq c (f(x^{0,j}) - f^*),  \]
which can be practically   verified. Using the second order growth
property  \eqref{wsc} we can also obtain easily linear convergence of the
generated sequence $x^{k}$ to some optimal point in $X^*$.

%%%%%%%%%%%%%%%%%%%%%%%%%%%%%%%%%%%%%%%%%%%%%%%%%%%%%%%
%%%%%%%%%%%%%%%%%%%%%%%%%%%%%%%%%%%%%%%%%%%%%%%%%%%%%%%%%

\subsection{Feasible descent methods (FDM)}
We now consider a more general descent  version of Algorithm (GM)
where the gradients are perturbed:
\begin{center}
\framebox{
\parbox{10cm}{
\begin{center}
\textbf{ Algorithm {\bf (FDM)} }
\end{center}
{Given $x^0  \in X$ and $\beta, L >0$  for $k \geq 0$ do:}
\begin{enumerate}
\item[] Compute ${x}^{k+1}=\left[ x^k - \alpha_k \nabla f(x^k) + e^k \right]_X$
\item[] such that
\item[]  $\|e^{k}\| \leq \beta \|x^{k+1} - x^k \|$ \; and \;  $f(x^{k+1}) \leq f(x^k) -
\frac{L}{2} \|x^{k+1} - x^k\|^2,$
\end{enumerate}
}}
\end{center}
where the stepsize $\alpha_k$ is chosen such that $\alpha_k \geq
\bar{L}_f^{-1} >0$ for all $k$. It has been showed in
\cite{LuoTse:92,WanLin:14} that algorithm (FDM) covers important
particular schemes: e.g.  proximal point minimization, random/cyclic
coordinate descent, extragradient descent and matrix splitting
methods are all feasible descent methods. Note that  linear
convergence of algorithm (FDM) under the error bound assumption
\eqref{geb}, i.e. $f \in \mathcal{E}_{L_f,\kappa_f}^{}$,  is
proved e.g.  in \cite{LuoTse:92,WanLin:14}. Hence, in the next
theorem we prove that the  feasible descent  method (FDM) converges
linearly in terms of function values on optimization problems (P)
whose objective functions belong to the class
$\mathcal{F}_{L_f,\kappa_f}^{}$.

\begin{theorem}
\label{the_fdm}
Let  the optimization  problem (P) have the objective function
belonging to the class $\mathcal{F}_{L_f,\kappa_f}^{}$. Then, the
sequence $x_k$ generated by the  feasible descent method (FDM) on
(P) converges linearly in terms of function values with the rate:
\begin{align}
\label{lin_conv_fdm_S} f(x^{k}) - f^*  \leq \left( \frac{1}{1 + \frac{L \kappa_f}{4(L_f + \bar{L}_f + \beta \bar{L}_f)^2}} \right)^k (f(x^0) - f^*).
\end{align}
\end{theorem}

\begin{proof}
The optimality conditions for computing $x^{k+1}$ are:
\begin{align}
\label{opt_pgm} \langle  x^{k+1} - x^k + \alpha_k \nabla f(x^k) - e^k, x -
x^{k+1} \rangle \geq 0 \quad \forall x \in X.
\end{align}
Then, using convexity of $f$  and Cauchy-Schwartz inequality, we get:
\begin{align*}
f(x^{k+1}) & -  f(\bar{x}^{k+1})  \leq \langle \nabla f(x^{k+1}),  x^{k+1} -  \bar{x}^{k+1} \rangle\\
& = \langle \nabla f(x^{k+1}) - \nabla f(x^{k}) + \nabla f(x^{k}),  x^{k+1} -  \bar{x}^{k+1} \rangle \\
&\overset{\eqref{lipg}+\eqref{opt_pgm}}{\leq} \!\!\! L_f \|x^{k+1} \!-  {x}^{k} \| \|x^{k+1} \!-
\bar{x}^{k+1}\| \!+\! \frac{1}{\alpha_k}   \langle x^{k+1} \!- x^{k} \!- e^k,  \bar{x}^{k+1} \!-  {x}^{k+1} \rangle \\
& \leq (L_f + \bar{L}_f) \|x^{k+1} -  {x}^{k} \| \|x^{k+1} -  \bar{x}^{k+1}\| + \bar{L}_f \|e^k\| \|x^{k+1} -  \bar{x}^{k+1}\| \\
& \leq (L_f + \bar{L}_f + \beta \bar{L}_f)
 \|x^{k+1} -  {x}^{k} \| \|x^{k+1} -  \bar{x}^{k+1}\|.
\end{align*}
Since $f \in \mathcal{F}_{L_f,\kappa_f}^{}$ then it satisfies  the
second order growth property, i.e. $ f(x^{k+1}) -  f(\bar{x}^{k+1})
\geq \frac{\kappa_f}{2} \|x^{k+1} -  \bar{x}^{k+1}\|^2$, and using
it  in the previous derivations we obtain:
\begin{align}
\label{pgm1}
f(x^{k+1}) -  f(\bar{x}^{k+1})  \leq  \frac{2 (L_f + \bar{L}_f + \beta \bar{L}_f)^2}{\kappa_f}  \|x^{k+1} -  x^{k} \|^2.
\end{align}
Combining \eqref{pgm1} with the descent property of the algorithm (FDM), i.e.  $
\|x^{k+1} - x^k\|^2 \leq \frac{2}{L} \left( f(x^{k}) - f(x^{k+1})
\right)$,  we get:
\[ f(x^{k+1}) -  f(\bar{x}^{k+1})  \leq   \frac{4 (L_f + \bar{L}_f + \beta \bar{L}_f)^2}{L  \kappa_f} \left( f(x^k)  - f(x^{k+1}) \right), \]
which leads to
\[  f(x^{k+1}) -  f(\bar{x}^{k+1})  \leq  \frac{1}{1 + \frac{L \kappa_f}{4(L_f + \bar{L}_f + \beta \bar{L}_f)^2}}  \left( f(x^k) - f(\bar{x}^k) \right). \]
Using an inductive argument we get the statement of the theorem. \qed
\end{proof}

\noindent Note that, once we have obtained linear convergence in terms of
function values for the algorithm (FDM), we can also obtain linear convergence of the
generated sequence $x^k$ to some optimal point in $X^*$ by using the
second order growth property  \eqref{wsc}.

%%%%%%%%%%%%%%%%%%%%%%%%%%%%%%%%%%%%%%%%%%%%%%%%%%%%%%%%%%%%%%%%%%%%%%%%%%%%%

\subsection{Discussions}

\noindent From previous sections we can conclude that for some classes of problems improved linear convergence  rates are obtained as compared to the existing results. For example,  in \cite{LuoTse:92,NecCli:14,WanLin:14} it has been proved that the optimization  problem \eqref{general_probl} whose objective function satisfies the conditions of Theorem \ref{the_gAx_cx} has on a compact set an error bound property  of the form \eqref{geb}. In this paper we  proved that this class of problems has the objective function satisfying the quadratic functional growth \eqref{wsc}. For the class of problems having an objective function satisfying  an error bound condition  the  feasible descent method (FDM) is shown to converge linearly in \cite{LiuWri:15,LuoTse:92,NecCli:14,WanLin:14,ZhaYin:13}. Note that for $\alpha_k=1/L_f, \beta=0$ and $L=L_f$ we recover from algorithm (FDM) the algorithm (GM). However, for these choices the linear convergence in \eqref{lin_conv_f}, given by $\frac{1}{1 +  \mu_f}$, is better than the one obtained in Theorem \ref{the_fdm}, given by  $\frac{1}{1 +  \mu_f/16}$.

\vspace{3pt}

\noindent Recently, in \cite{ZhaYin:13} the authors show that the class of convex
unconstrained problems $\min_{x \in \rset^n} g(Ax)$, with $g$ strongly convex function having Lipschitz continuous gradient,   satisfies a restricted strong convex
inequality, which is a particular version of our  more general  quadratic gradient growth inequality \eqref{gvar}. However, in this paper we proved  that the objective function of this particular class of optimization problems belongs to a more restricted functional class, namely  $q\mathcal{S}_{L_f,\kappa_f}^{}(X)$, i.e. it satisfies \eqref{wsc_basic}.  Thus, for this class of problems we provide better linear rates for gradient method and for fast gradient method as compared to \cite{ZhaYin:13}. More precisely, for the gradient method (GM) we derive convergence rate of order  $(1-\mu_f)/(1+\mu_f)$, while \cite{ZhaYin:13} proved convergence rate  of order $(1 - \mu_f)$. Moreover, from our best knowledge, this paper shows for the first time linear convergence of the usual fast gradient method (FGM) for this class of convex problems $\min_{x \in \rset^n} g(Ax)$, while  for example \cite{ZhaYin:13}
derives a worse rate of convergence and  for a restarting variant of the fast gradient method (R-FGM).

%\begin{remark}
%\noindent For all $i=1:n$, let $e_i$ denote the standard basis vector for coordinate $i$ and let $x_i = e_i^T x$ and $\nabla_i f(x) = e_i^T \nabla f(x)$.   Let us  also consider the particular case:  $X = \prod_{i=1}^n X_i$ and the objective function $f(x_1,\cdots,x_n)$ of problem (P) has $L_i$-coordinatewise Lipschitz continuous gradient (i.e. $\|\nabla_i f(x + t e_i) - \nabla f(x)\| \leq L_i |t|$ for all $x, x+ t e_i \in X$) and satisfies  the second order growth property \eqref{wsc}. Then, using similar arguments as in \cite{NecCli:14}, we can prove that the random coordinate gradient descent method:
%\[ x_i^{k+1} =\left[ x_i^k - \frac{1}{L_i} \nabla_i f(x^k) \right]_{X_i}  \quad \text{and}
%\quad x_j^{k+1} = x_j^k \;\; \forall i \not= j,  \] where the index
%$i$ is chosen uniformly random, has linear convergence in the
%expected values of the objective function. Note that in this case the global Lipschitz constant $L_f$ satisfies: $L_f \leq \sum_{i=1}^n L_i$.   \qed
%\end{remark}

%%%%%%%%%%%%%%%%%%%%%%%%%%%%%%%%%%%%%%%%%%%%%%%%%%%%%%%%%%%%%%%%%%%%%%
%%%%%%%%%%%%%%%%%%%%%%%%%%%%%%%%%%%%%%%%%%%%%%%%%%%%%%%%%%%%%%%%%%%%%%

\section{Applications}
\label{sec_applications}
In this section we present several  applications having the objective function in one of the structured functional classes of Section \ref{sec_gAx}.

\subsection{Solution of linear systems}
It is well known that finding a solution of a symmetric linear
system $Qx+q=0$,  where $Q \succeq 0$ (notation for positive
semi-definite matrix), is equivalent to solving a convex quadratic
program (QP):
\[  \min_{x \in \rset^n} f(x)  \qquad \left( =\frac{1}{2} x^T Q x + q^T x \right).  \]
Let $Q = L_Q^T L_Q$ be  the Cholesky decomposition of $Q$.  For
simplicity,  let us assume that our symmetric linear system has a
solution, e.g. $x_s$, then $q$ is in the range of $Q$, i.e. $q = -Q
x_s = - L_Q^T L_Q x_s$. Therefore, if we define the strongly convex
function $g(z) = \frac{1}{2} \|z\|^2 -  (L_Q x_s)^T z$, having
$L_g=\sigma_g=1$, then our objective function is the composition of
$g$ with the linear map $L_Qx$:
\[  f(x) = \frac{1}{2} \|L_Qx\|^2 - (L_Q^T L_Q x_s)^T x = g(L_Qx). \]
Thus, our convex quadratic problem  is  in the form of unconstrained
structured  optimization problem \eqref{generalcls} and  from
Section \ref{sec_gAx} we conclude that the objective function of
this QP is in the class $q\mathcal{S}_{L_f,\kappa_f}^{}$ with:
\[  L_f = \lambda_{\text{max}}(Q) \; \text{and} \; \kappa_f \!=\! \sigma_{\text{min}}^2(L_Q) \!=\! \lambda_{\text{min}}(Q) \quad \Rightarrow \quad \mu_f \!=\! \frac{\lambda_{\text{min}}(Q)}{\lambda_{\text{max}}(Q)}  \!\equiv\!  \frac{1}{\text{cond}(Q)}, \]
where $\lambda_{\text{min}}(Q)$ denotes the smallest non-zero eigenvalue of $Q$ and $\lambda_{\text{max}}(Q)$ is the largest eigenvalue of $Q$. Since we assume that our symmetric linear system has a solution, i.e. $f^*=0$, from Theorem \ref{th_fgm_W} and Remark \ref{remark_all_bary} we conclude that when solving this convex  QP with the algorithm (FGM)  we get the convergence rate in terms of function values:
\[ f( x^k) \leq \left( 1 - \sqrt{\frac{1}{\text{cond}(Q)}} \right)^k  \cdot 2 f(x^0)   \]
or in terms of residual (gradient) or distance to the solution:
\begin{align*}
\| Qx^k + q\|^2 &= \|\nabla f(x^k)\|^2 \leq L_f^2 \| x^k - \bar{x}^k\|^2 \leq \frac{2L_f^2}{\kappa_f} \left( f(x^k) - f^* \right) \\
& \leq \left( 1 - \sqrt{\frac{1}{\text{cond}(Q)}} \right)^k  \cdot \lambda_{\text{max}}(Q) \cdot \text{cond}(Q) \left( \frac{1}{2}(x^0)^T Q x^0  + q^T x^0 \right).
\end{align*}
Therefore, the usual (FGM) algorithm without restart  attains an
$\epsilon$ optimal solution in a number of iterations of order
$\sqrt{\text{cond}(Q)} \log \frac{1}{\epsilon}$, i.e. the condition
number $\text{cond}(Q)$ of the matrix $Q$ is square rooted.  From
our knowledge, this is one of the first results showing linear
convergence depending on the square root of the condition number
for the  fast gradient method on solving a symmetric linear
system with positive semi-definite matrix $Q \succeq 0$.  Note that
the linear conjugate gradient method can also attain an $\epsilon$
approximate solution in much fewer than $n$ steps, i.e. the same
$\sqrt{\text{cond}(Q)} \log \frac{1}{\epsilon}$ iterations
\cite{Bub:15}. Usually, in the literature the condition number appears linearly in the convergence
rate of first order methods for solving linear systems with
 positive semi-definite matrices. For example,  the coordinate
descent method from  \cite{LevLew:10} requires $\sqrt{n} \cdot
\text{cond}(Q) \log \frac{1}{\epsilon}$ iterations for obtaining an
$\epsilon$ optimal solution.

\noindent Our results can be extended for solving general linear systems $Ax+b=0$, where $A \in \rset^{m \times n}$.  In this case we can formulate
the equivalent unconstrained optimization problem:
\[ \min_{x \in \rset^n} \| Ax+b \|^2 \]
which is  a particular case of \eqref{generalcls} and  from Section
\ref{sec_gAx} we can also conclude that the objective function of
this QP is in the class $q\mathcal{S}_{L_f,\kappa_f}$ with:
\[  L_f = \sigma_{\text{max}}^2(A) \;\; \text{and} \;\; \kappa_f \!=\! \sigma_{\text{min}}^2(A)
\quad \Rightarrow \quad \mu_f =
\frac{\sigma_{\text{min}}^2(A)}{\sigma_{\text{max}}^2(A)}, \] where
$\sigma_{\text{min}}(A)$ denotes the smallest non-zero singular
value of $A$ and $\sigma_{\text{max}}(A)$ is the largest singular
value of $A$.  In this case the usual (FGM) algorithm  attains and
$\epsilon$ optimal solution in a number of iterations of order
$\frac{\sigma_{\text{max}}(A)}{\sigma_{\text{min}}(A)} \log
\frac{1}{\epsilon}$.

%%%%%%%%%%%%%%%%%%%%%%%%%%%%%%%%%%%%%%%%%%%%%%%%%%%%%%%%%%%%%%%%%%%%%%%%%5

\subsection{Dual of linearly constrained convex problems}
Let (P) be the dual formulation of a linearly constrained convex
problem:
\begin{align*}
& \min_{u} \tilde{g}(u) \\
& \text{s.t.}: \; c - A^T u  \in \mathcal{K} = \rset^{n_1} \times \rset^{n_2}_{+}.
\end{align*}
Then, the  dual of this optimization problem can be written in the form of structured problem \eqref{general_probl}, where $g$ is the convex conjugate of $\tilde{g}$. From duality theory we know that $g$ is strongly convex and with Lipschitz gradient, provided that $\tilde{g}$ is strongly convex and with Lipschitz gradient.

%%%%%%%%%%%%%%%%%%%%%%%%%%%%%%%%%%%%%%%%%%%%%%%%%%%%%%%%%%%%%%%%%%%%%%%%%

\subsection{Lasso problem}
The Lasso problem is defined as:
\[ \min_{x: Cx \leq d} \|Ax - b\|^2 + \lambda \|x\|_1. \]
Then, the Lasso problem is a particular case of the structured optimization  problem \eqref{general_probl}, provided that e.g. the feasible set of this problem is bounded (polytope).

%%%%%%%%%%%%%%%%%%%%%%%%%%%%%%%%%%%%%%%%%%%%%%%%%%%%%%%%%%%%%%%%%%%%%%%%%5

\subsection{Linear  programming}
\label{sect_LP} Finding a primal-dual solution of a linear cone
program can also be  written in the form of a structured
optimization problem \eqref{generalcls}. Indeed, let $c \in \rset^N,
b \in \rset^m$ and $\mathcal{K} \subseteq \rset^{N}$ be a closed
convex cone, then we define the linear  cone programming:
\begin{align}
\label{coneprob_primal}
 \min\limits_{u} & \; \langle c, u \rangle \qquad   \text{s.t.} \;\;  Eu = b, \quad u \in \mathcal{K},
\end{align}
and its associated dual problem
\begin{align}
\label{coneprob_dual}
 \min\limits_{v, s} &\; \langle b, v\rangle \qquad
  \text{s.t.} \;\;   E^Tv + s = c,\quad s \in \mathcal{K}^*,
\end{align}
where $\mathcal{K}^*$ denotes the dual cone.  We assume that the
pair of cone programming
\eqref{coneprob_primal}--\eqref{coneprob_dual} have optimal
solutions and their associated duality gap is zero. Therefore, a
primal-dual solution of
\eqref{coneprob_primal}--\eqref{coneprob_dual} can be found by
solving the following convex feasibility problem, also called
homogeneous self-dual embedding:
\begin{equation}\label{optcond_cone}
\text{find} \; (u,v,s) \; \text{such that} \;\;
\begin{cases}
 E^Tv + s = c, \;\; Eu =b, \;\; \langle c,u\rangle = \langle b,v\rangle \\
u \in \mathcal{K}, \;\; s \in \mathcal{K}^*, \;\; v \in \rset^{m},
\end{cases}
\end{equation}
or, in a more compact formulation:
$$\text{find} \; x \; \text{such that} \;\; \begin{cases} Ax = d \; \\
x  \in {\mathcal{\textbf{K}}},
\end{cases}
 $$
where $x = \begin{bmatrix}
                       u \\ v \\ s
                    \end{bmatrix}, \;\;
A = \begin{bmatrix}
     0 &E^T &I_n \\
     E &0 &0 \\
     c^T &-b^T &0
    \end{bmatrix}, \;\;
d = \begin{bmatrix}
     c \\
     b \\
     0
    \end{bmatrix}$, \;\;
$\mathcal{\textbf{K}}= \mathcal{K}\times \rset^m \times
\mathcal{K}^*$. The authors in \cite{DonChu:16} proposed solving
conic optimization problems in homogeneous self-dual embedding form
using ADMM. In this paper we propose solving a linear program in the
homogeneous self-dual embedding form using the first order methods
presented above. A simple reformulation of this constrained linear
system as an optimization problem is:
\begin{align}
\min\limits_{x  \in \mathcal{\textbf{K}}} &\; \norm{Ax -d}^2.
\label{cone_opt}
\end{align}
Denote the dimension of the variable $x$ as $n = 2N+m$.  Let us note
that the  optimization problem \eqref{cone_opt} is a particular case
of \eqref{generalcls} with objective function of the form $f(x) = g(Ax)$, with
$g(\cdot) = \|\cdot - d\|^2$. Moreover, the conditions of Theorem
\ref{the_gAx} hold   provided that  $\mathcal{K} = \rset^N_+$.   We
conclude that we can always solve a linear program in linear time
using  the first order methods described in the present paper.

%%%%%%%%%%%%%%%%%%%%%%%%%%%%%%%%%%%%%%%%%%%%%%%%%%%%%%%%%%%%%%%%%%%%%%%

\subsection{Numerical simulations}
We test the performance of first order algorithms described above on
randomly generated Linear Programs \eqref{coneprob_primal} with
$\mathcal{K} = \rset^N_+$. We assume Linear Programs with finite
optimal values. Then, we can  reformulate \eqref{coneprob_primal} as
the quadratic convex problem \eqref{cone_opt} for which $f^*=0$. We
compare the following algorithms for problem \eqref{cone_opt} (the results are given in Figures 1 and 2):
\begin{itemize}
\item[1.] Projected gradient algorithm  with fixed stepsize (GM):  $\alpha_k =
\|A\|^{-2}$ (in this case the Lipschitz constant is $L_f=\|A\|^2$).

\item[2.] Fast gradient algorithm with restart (R-FGM): where $c =
10^{-1}$ and we restart  when $\|Ax^{K_c^*,j} - d\| \leq c \|
Ax^{0,j} - d \|$.

\item[3.] Exact cyclic coordinate descent algorithm (Cyclic CD):
\[ x_i^{k+1} = \arg \min_{x_i \in {\mathcal{\textbf{K}}}_i} \|A x_i(k) - d\|^2,  \]
where ${\mathcal{\textbf{K}}}_i$ is either $\rset_+$ or $\rset$ and
$x_i(k) =[x_1^{k+1} \cdots x_{i-1}^{k+1} \; x_i \; x_{i+1}^k \cdots
x_n^k]$. It has been proved in \cite{LuoTse:92} that this algorithm
is a particular version of the feasible descent method (FDM)
with parameters:
\[ \alpha_k = 1, \qquad \beta = 1+ L_f \sqrt{n}, \qquad L = \min_i \|A_i\|^2,  \]
provided that all the columns  of $A$ are nonzeros, i.e. $\|A_i\| >
0$ for all $i=1:n$.
\end{itemize}

\noindent The comparisons use  Linear Programs  whose data $(E, b,
c)$ are  generated randomly from the standard Gaussian distribution
with full or sparse matrix $E$. Matrix $E$ has $100$ rows and $150$
columns in the full case and  $900$ rows and $1000$ columns in the
sparse case. %All iterations were stopped when
%$\frac{\|A x^k - d\|}{\|d\|} \leq 10^{-6}$. For cyclic coordinate
%descent we count the number of full iterations (the number of groups
%by $n$ coordinate iterations).
Figures 1 and 2 depict the error $\|A x^k -d\|$. We can observe that
the gradient method has a slower convergence  than the  fast gradient method with restart,
but both have a linear behaviour as we can see from the comparison with the  theoretical
sublinear estimates, see Figure 1. Moreover, the fast gradient method with
restart is performing much faster than the gradient or cyclic
coordinate descent methods on sparse and full Linear Programs, see Figure 2.

\begin{figure}
\label{fig_lcth} \caption{Linear convergence of algorithms \textbf{R-FGM} (left) and
\textbf{GM} (right): log scale of the error $\|Ax^k-d\|^2$. We also compare
with the  theoretical sublinear estimates (dot lines) for the
convergence rate  of algorithms \textbf{FGM} ($\mathcal{O} (L_fR^2_f/k^2)$) and
\textbf{GM} ($\mathcal{O} (L_fR^2_f/k)$) for smooth convex problems.
The plots clearly show  our theoretical findings, i.e. linear convergence.}
\includegraphics[height=5.5cm,width=6.3cm]{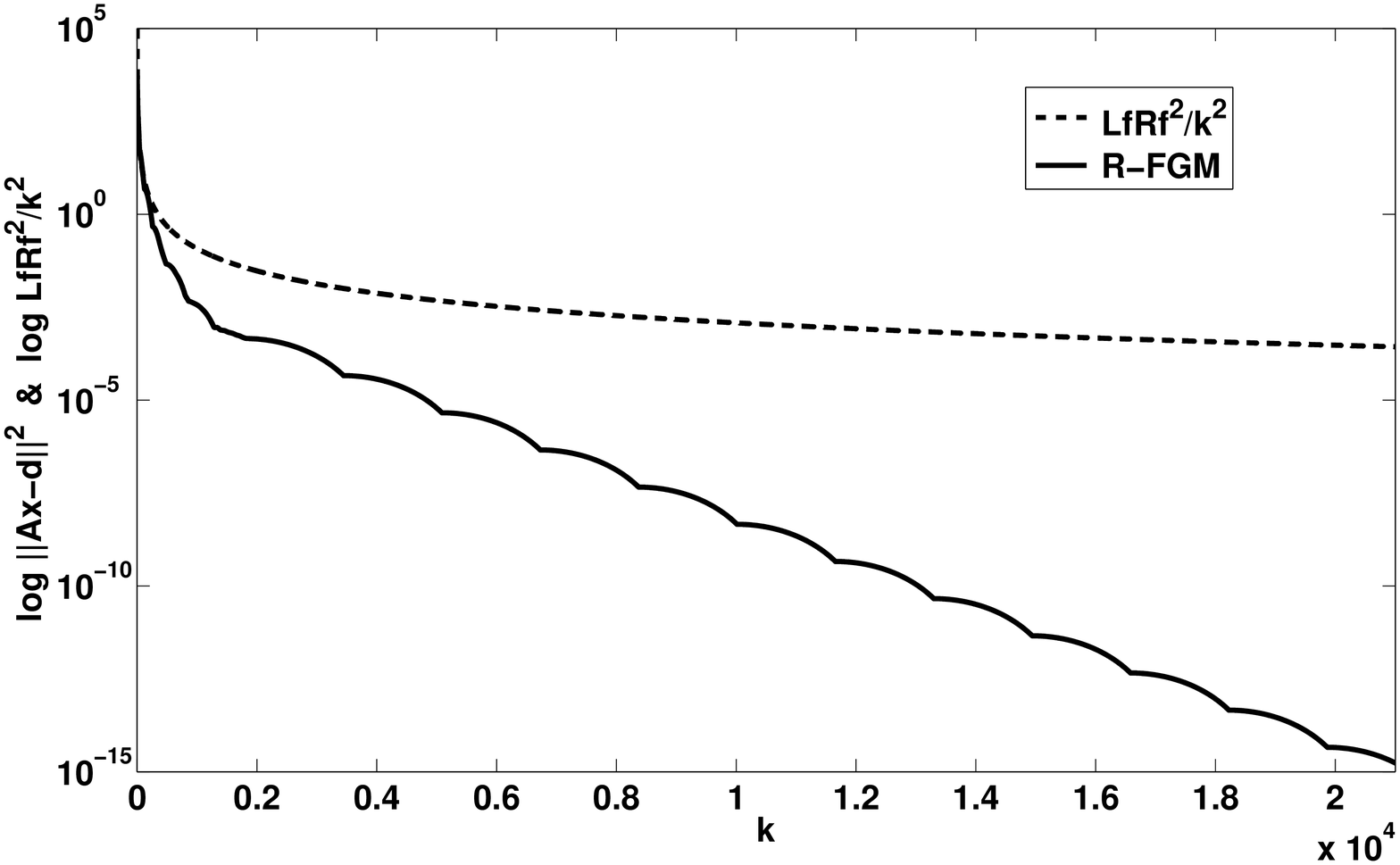}
\hspace*{-0.5cm}
\includegraphics[height=5.5cm,width=6.3cm]{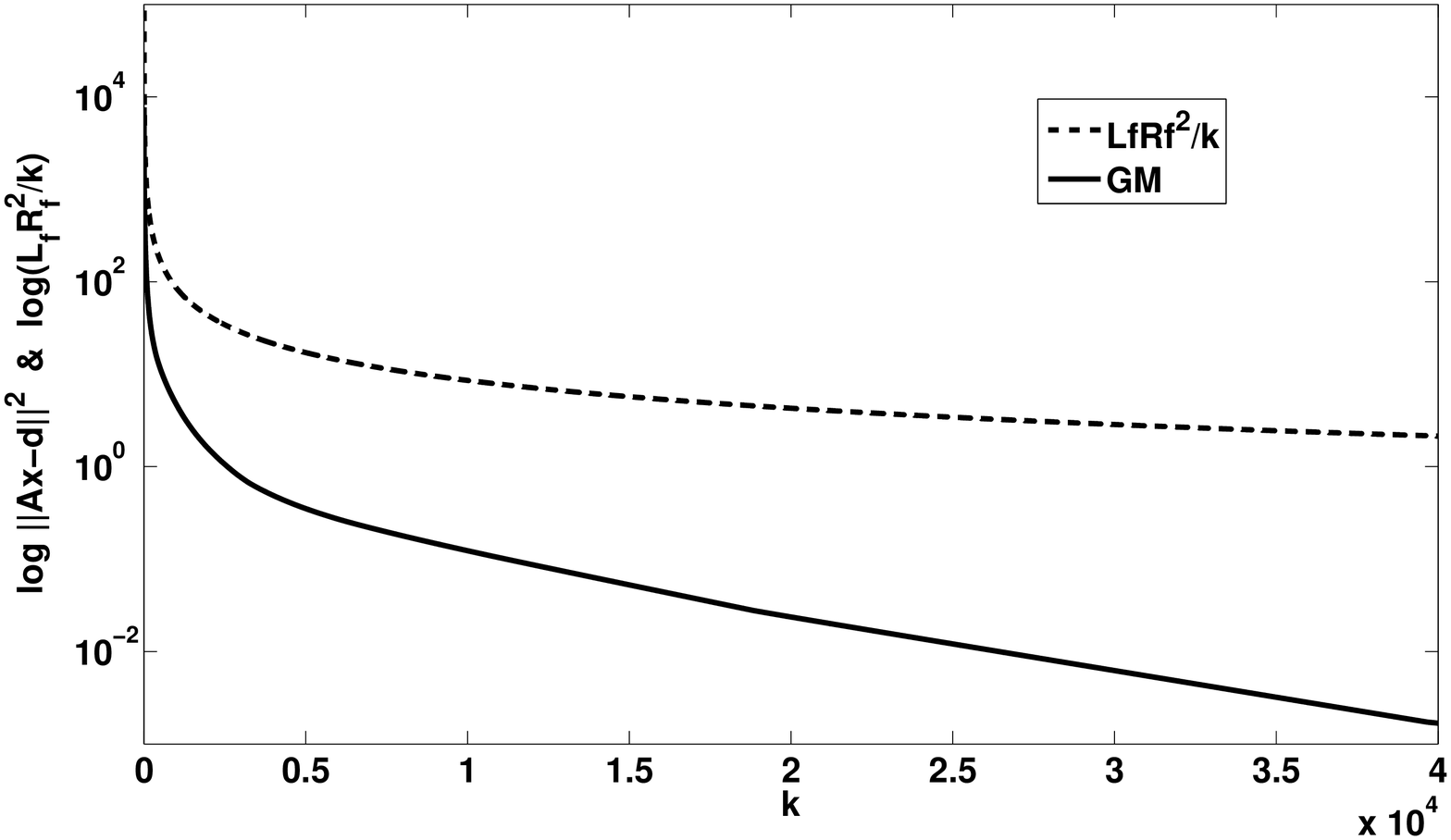}
\end{figure}

\begin{figure}
\label{fig_lcth} \caption{The behavior  of algorithms  \textbf{GM},
\textbf{R-FGM} and \textbf{Cyclic CD}: log scale of the error
$\|Ax^k-d\|$ along iterations $k$ (left - full $E$, right - sparse
$E$).}
\includegraphics[height=5.5cm,width=6.3cm]{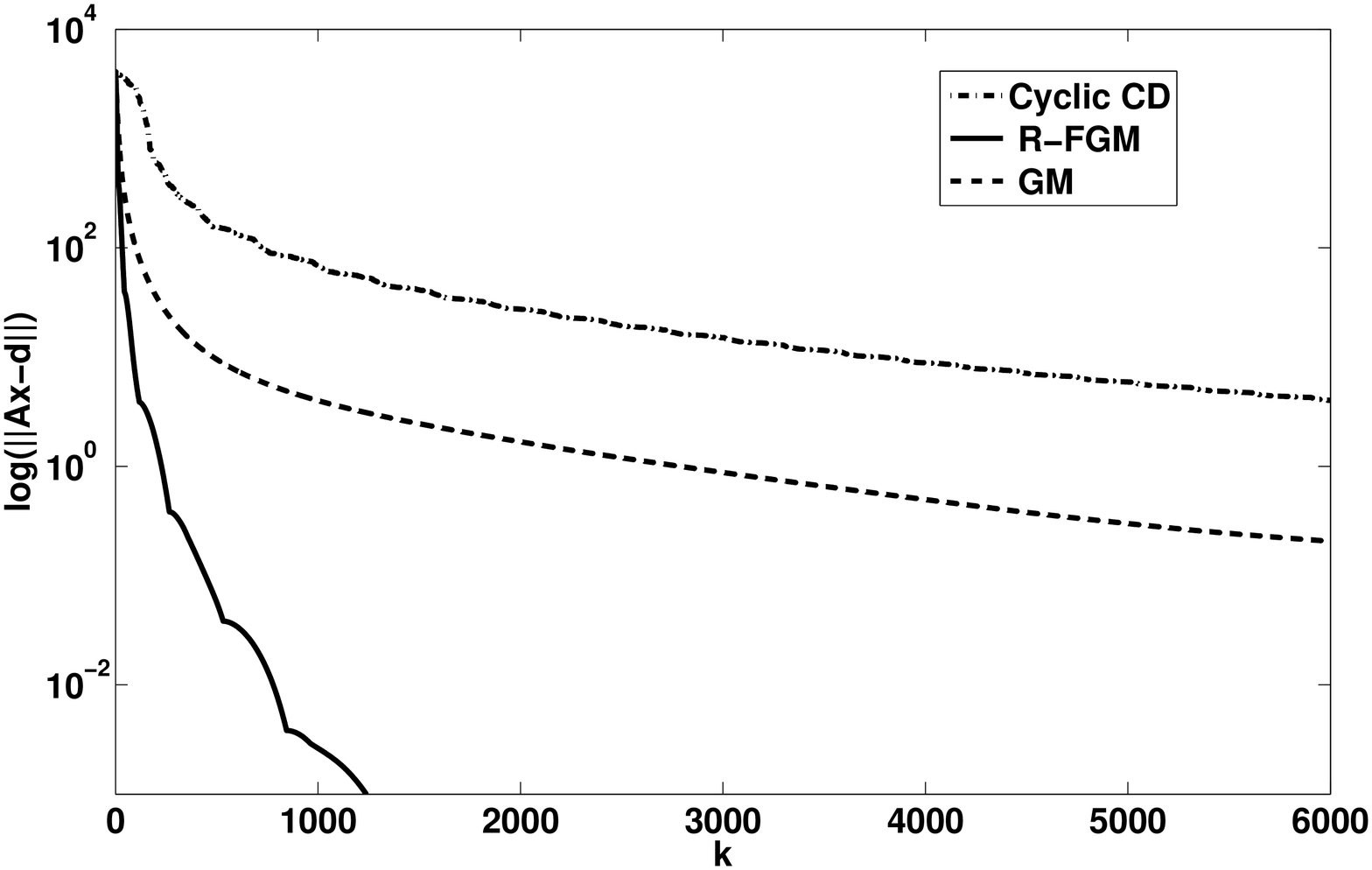}
\hspace*{-0.5cm}
\includegraphics[height=5.5cm,width=6.3cm]{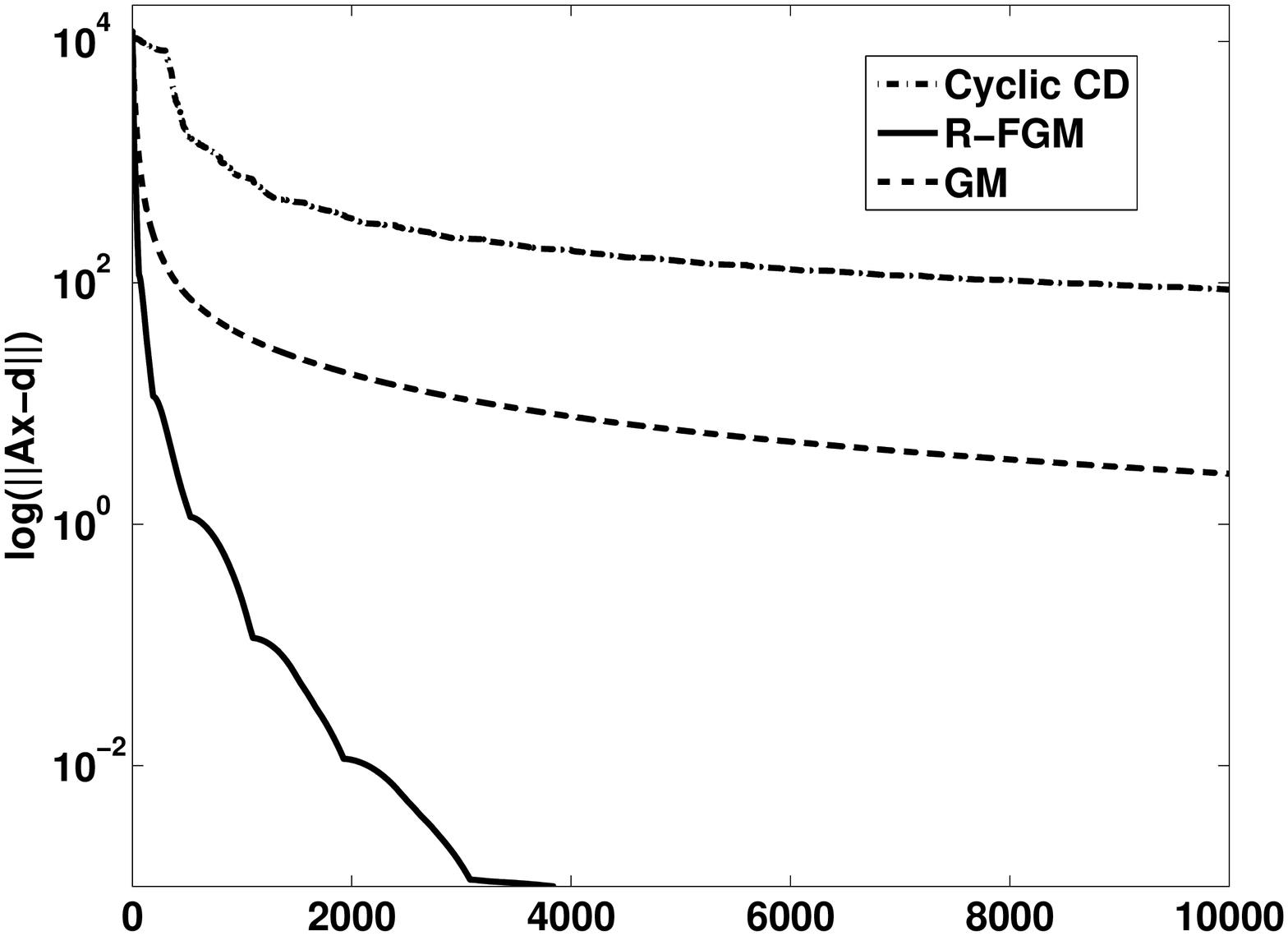}
\end{figure}

\end{document}